\documentclass[10pt]{amsart}

\usepackage{hyperref}
\usepackage{enumerate}
\usepackage{comment}

\makeatletter

\@namedef{subjclassname@2010}{

  \textup{2010} Mathematics Subject Classification}

\usepackage{amssymb, amsmath,amsthm}
\usepackage{mathrsfs}
\newtheorem{thm}{Theorem}[section]

\newtheorem{prop}[thm]{Proposition}
\newtheorem{conj}[thm]{Conjecture}

\newtheorem{cor}[thm]{Corollary}
\newtheorem{lem}[thm]{Lemma}
\theoremstyle{definition}
\newtheorem{rem}[thm]{Remark}
\newtheorem{def1}[thm]{Definition}

\newcommand{\ra}{\rightarrow}
\newcommand{\bk}{\backslash}
\newcommand{\mc}{\mathcal}
\newcommand{\mf}{\mathfrak}
\newcommand{\mb}{\mathbb}

\renewcommand{\ss}{\substack}

\newcommand{\e}{\varepsilon}

\newcommand{\mbf}{\boldsymbol}

\renewcommand{\bar}{\overline}

\begin{document}
\title{On S\'{a}rk\"{o}zy's Local Extrema Conjectures}
\author{Alexander P. Mangerel}
\address{Department of Mathematical Sciences, Durham University, Stockton Road, Durham, DH1 3LE, UK}
\email{smangerel@gmail.com}
\maketitle
\begin{abstract}
Let $f: \mb{N} \ra (0,\infty)$ be a multiplicative function, and let $\omega > 0$. We consider the frequency with which patterns of the form
$$
f(n) > \omega \max\{f(n-1),f(n+1)\}, \quad f(n) < \frac{1}{\omega}\min\{f(n-1),f(n+1)\}
$$
occur in the sequence $(f(n))_n$, and give necessary conditions on $f$ so that the set of occurrences of at least one of these patterns has  logarithmic density $0$. This includes a complete classification of all such completely multiplicative functions. 
As a particular application, we characterise (up to prime powers $p^k$, $k \geq 2$) all multiplicative functions $f: \mb{N} \ra \mb{N}$ that have either finitely many local maxima, or finitely many local minima. This answers, in a strong form, two conjectures of A. S\'{a}rk\"{o}zy from 2001. \\
The proof is based on a natural ``pretentiousness'' dichotomy for additive functions. Key tools in the analysis include an analogue of the Matom\"{a}ki-Radziwi\l\l{} theorem for additive functions, due to the author, as well as a new bound for the concentration of the gaps $g(n+1)-g(n)$ of an additive function $g$, which may be of independent interest.  
\end{abstract}
\section{Introduction}
For a parameter $\omega > 0$ and an arithmetic function $f: \mb{N} \ra \mb{R}$, define the sets
\begin{align*}
\mc{L}_f(\omega) &:= \{n \in \mb{N} : \ f(n) > \omega \max\{f(n-1),f(n+1)\}\}, \\ \mc{S}_f(\omega) &:= \{n \in \mb{N} : \ f(n) < \frac{1}{\omega} \min\{f(n-1),f(n+1)\}\}.
\end{align*}
Set also $\mc{L}_f := \mc{L}_f(1)$ and $\mc{S}_f := \mc{S}_f(1)$. It is easily observed that $\mc{L}_f$ and $\mc{S}_f$ respectively correspond to the sets of local maxima and local minima of $f$. 
Our objective in this note is to study two conjectures of A. Sark\H{o}zy \cite{Sarkozy} concerning the size of the sets $\mc{L}_f(\omega)$ and $\mc{S}_f(\omega)$, when $f$ is a  multiplicative function taking positive integer values. \\
Sark\H{o}zy's first conjecture gives an exact classification of those multiplicative functions that have either finitely many local maxima, or finitely many local minima.
\begin{conj}[\cite{Sarkozy}, Conjecture 53] \label{conj:53}
Let $f: \mb{N} \ra \mb{N}$ be a multiplicative function such that for each $k \in \mb{N}_0$ there are at least two primes $p_1,p_2$ such that $f(p_j) \neq p_j^k$, $j = 1,2$. Then $\mc{L}_f$ and $\mc{S}_f$ are both infinite sets.
\end{conj}
His second conjecture, which deals with more general values of $\omega>0$, is as follows.
\begin{conj}[\cite{Sarkozy}, Conjecture 54]\label{conj:54}
Let $f: \mb{N} \ra \mb{N}$ be a multiplicative function such that for each $k \in \mb{N}_0$ there are infinitely many primes $p$ such that $f(p) \neq p^k$. Then for every $\omega > 0$ the sets $\mc{L}_f(\omega)$ and $\mc{S}_f(\omega)$ are both infinite.
\end{conj}
\begin{rem} \label{rem:MonoOmega}
Observe that for $\omega_1 \geq \omega_2$ we have
$$
\mc{L}_f(\omega_1) \subseteq \mc{L}_f(\omega_2), \quad \mc{S}_f(\omega_1) \subseteq \mc{S}_f(\omega_2).
$$
Therefore, Conjecture \ref{conj:54} with $\omega \in (0,1]$ follows immediately from Conjecture \ref{conj:53}. 
\end{rem}
In the cases $0 < \omega \leq 1$ we settle Conjecture \ref{conj:54} (and therefore also Conjecture \ref{conj:53}, by Remark \ref{rem:MonoOmega}) in a strong form. To state our result, we recall that the \emph{upper} and \emph{lower logarithmic densities} of a set $\mc{A} \subseteq \mb{N}$ are respectively defined as
$$
\overline{\delta}(\mc{A}) := \limsup_{x \ra \infty} \frac{1}{\log x} \sum_{n \in \mc{A} \cap [1,x]} \frac{1}{n}, \quad \underline{\delta}(\mc{A}) := \liminf_{x \ra \infty} \frac{1}{\log x} \sum_{n \in \mc{A} \cap [1,x]} \frac{1}{n},
$$
and that $\mc{A}$ is said to have \emph{logarithmic density} $\delta(\mc{A})$ if $\overline{\delta}(\mc{A}) = \underline{\delta}(\mc{A}) = \delta(\mc{A})$.
\begin{thm}[Conjecture \ref{conj:53} is true]\label{thm:multAtMost1}
Let $f: \mb{N} \ra \mb{N}$ be a multiplicative function such that for each $k \in \mb{N}_0$ there are at least two primes $p_1,p_2$ such that $f(p_j) \neq p_j^k$ for $j = 1,2$. Then for each $\omega \in (0,1]$, both $\mc{L}_f(\omega)$ and $\mc{S}_f(\omega)$ have positive upper logarithmic density. In particular, $\mc{L}_f$ and $\mc{S}_f$ are both infinite.
\end{thm}
On the other hand, it is actually quite easy to show that, as formulated, the conjecture is \emph{false} for general multiplicative functions, whenever $\omega > 1$. We do this in Section \ref{sec:2} below.
\begin{prop}[Conjecture \ref{conj:54} is false for all $\omega > 1$]\label{prop:counter}
For each $\omega >1$ there exist multiplicative functions $f_1,f_2: \mb{N} \ra \mb{N}$ such that for both $j = 1,2$ and for each $k \in \mb{N}_0$ there are infinitely many primes $p$ such that $f_j(p) \neq p^k$, yet $|\mc{S}_{f_1}(\omega)| < \infty$ and $|\mc{L}_{f_2}(\omega)| < \infty$. Indeed, there exists $A = A(\omega)$ such that for each prime $p_0$ such functions $f_j$ can be constructed with the following properties: 
\begin{enumerate}[(a)]
\item there are constants $C_1,C_2 > 0$ such that
$$
C_1 n^A \leq f_j(n) \leq C_2 n^A \text{ for all } p_0 \nmid n,
$$
\item we have
$$
\sup_{k \geq 1} f_1(p_0^k)/p_0^{kA} = \infty, \quad \inf_{k \geq 1} f_2(p_0^k)/p_0^{kA} = 0. 
$$
\end{enumerate}
\end{prop}
We will show that when $\omega > 1$, the properties (a) and (b) highlighted in the above proposition are in fact \emph{necessary} in order for $\mc{S}_f(\omega)$ and $\mc{L}_f(\omega)$ to have logarithmic density zero. To this end, we introduce the following definition.
\begin{def1}[Powerfully bad primes] 
Given a positive-valued multiplicative function $f$, we say that a prime $p$ is \emph{powerfully bad for $f$} if 
$$
\sup_{k \geq 1} |\log(f(p^k))| = \infty. 
$$
We define $\mc{B}_f$ to denote the set of all powerfully bad primes for $f$. 
\end{def1}
\begin{rem}
We note that when $f$ is completely multiplicative
we have
$$
\mc{B}_f = \{p : f(p) \neq 1\}.
$$
\end{rem}
In the sequel, given $\lambda \in \mb{R}$ and a multiplicative function $f$, write $f_{\lambda}(n) := f(n)/n^{\lambda}$ for all $n$. 
Our main result, which treats the more general class of positive-valued multiplicative functions, is as follows.

\begin{thm}[At most one prime is powerfully bad]\label{thm:PowBad}
Let $f: \mb{N} \ra (0,\infty)$ be a multiplicative function. Let $\omega > 0$, and suppose that $\delta(\mc{S}_f(\omega))\delta(\mc{L}_f(\omega)) = 0$. Then there exists $\lambda \in \mb{R}$ such that $|\mc{B}_{f_{\lambda}}| \leq 1$. Moreover, there exists an integer $P$ with $\Omega(P) \leq 1$ such that
$$
\sup_{(n,P) = 1} |\log(f(n)/n^\lambda)| < \infty.
$$
\end{thm}
\begin{rem}
In light of the example from the proof of Proposition \ref{prop:counter}, this classification is sharp.
\end{rem}
When $f$ is completely multiplicative, in particular, we obtain the following strong form of Conjecture \ref{conj:54} for all $\omega > 0$ as a consequence of Theorems \ref{thm:PowBad} and \ref{thm:multAtMost1}. 
\begin{thm}[Conjecture \ref{conj:54} holds for completely multiplicative functions] \label{thm:SarkCM}
If $f: \mb{N} \ra \mb{N}$ is a completely multiplicative function such that for each $k \in \mb{N}_0$ there are at least two primes $p_1,p_2$ such that $f(p_j) \neq p_j^k$ for $j = 1,2$, then for each $\omega > 0$, both $\mc{L}_f(\omega)$ and $\mc{S}_f(\omega)$ have positive upper logarithmic density.
\end{thm}
Note that this theorem improves on the conjectured claim in two aspects: the number of exceptional primes has been reduced to $2$, and the infinitude of $\mc{L}_f(\omega)$ and $\mc{S}_f(\omega)$ has been quantified in terms of logarithmic upper density. \\
In the author's opinion, one compelling feature of our main theorems is that despite our incomplete understanding of the theory of 3-point correlations of multiplicative functions (on which there has nevertheless been spectacular progress in the work of Tao, Ter\"{a}v\"{a}inen and their collaborators, see e.g., \cite{TTDuke}, \cite{KMT2}, \cite{MRTTZ}), we are able to give quantitative information about patterns of three additively related values of $f$.  Moreover, the size of the values of $f$ turn out not to pose an obstacle to this objective.
\subsection{Background and proof strategy}
\subsubsection{Initial setup}
For most of this paper we will work more generally with multiplicative functions $f: \mb{N} \ra (0,\infty)$ that take positive real values. We note that this set is invariant under the involution $f \mapsto 1/f$, and that for every $\omega > 0$,
\begin{equation}\label{eq:ftoreciprocal}
\mc{S}_f(\omega) = \mc{L}_{1/f}(\omega).
\end{equation}
Thus, replacing $f$ by $1/f$, we may concern ourselves mainly with the question of when $\delta(\mc{L}_f(\omega)) = 0$, which we shall do in the sequel. \\
The starting point for our method comes from the obvious fact that $g := \log f$ is a real-valued additive function. By Remark \ref{rem:MonoOmega}, we may restrict our attention to $\omega \geq 1$. If we set $\Delta := \log \omega \geq 0$, then $n \in \mc{L}_f(\omega)$ if and only if
\begin{align*}
g(n) > \max\{g(n-1),g(n+1)\} + \Delta.
\end{align*}
We fix this notation for the remainder of the paper. These observations reduce matters to studying a local property of real-valued additive functions. 
\subsubsection{Historical overview}
At this point it is worth giving historical background on rigidity problems for additive functions, which has been a topic of significant interest for a long time. A major influence on this topic was a classical paper of Erd\H{o}s \cite{Erdos}, which gave examples of several local rigidity properties that can only be satisfied by constant multiples of the logarithm, i.e., $g(n) = c \log n$ for $c \in \mb{R}$. Among many others in \cite{Erdos}, he proved two important theorems of this type:
\begin{enumerate}[(i)]
\item if $g(n+1)-g(n) = o(1)$ as $n \ra \infty$ then there is a constant $c\in \mb{R}$ such that $g = c\log$; 
\item if $g(n+1) \geq g(n)$ for all $n$, then there exists $c \in \mb{R}$ such that $g = c\log$.
\end{enumerate}
Various generalisations of these types of questions have also been considered in the literature; see e.g. \cite{Mangerel} for a representative sample of references. \\
In \cite{Erdos}, Erd\H{o}s also conjectured that the following strengthening of (i)-(ii) may be obtained by relaxing the corresponding local properties therein:
\begin{enumerate}[(i')]
\item if $g(n_k+1) - g(n_k) = o_{k \ra \infty}(1)$ along a sequence $(n_k)_k$ of density $1$ then there is a constant $c\in \mb{R}$ such that $g = c\log$; 
\item if $g(n_k+1) \geq g(n_k)$ along a sequence $(n_k)_k$ of density $1$, then there exists $c \in \mb{R}$ such that $g = c\log$.
\end{enumerate}
The first of these conjectures was proven in 1988 by Hildebrand \cite{Hildebrand}. The second of these resisted proof for much longer. \\
Within the last decade especially, after the breakthrough by Matom\"{a}ki and Radziwi\l\l{} on multiplicative functions in short intervals \cite{MR}, it was observed that local properties of multiplicative functions, and therefore also of additive functions, can be studied using suitable estimates for the partial sums of these functions over short intervals. Indeed, in \cite{Mangerel} the author made partial progress on problem (ii') by deriving a new short-interval mean value theorem for general additive functions from the Matom\"{a}ki-Radziwi\l\l{} theorem. Recently, an AI-assisted proof of problem (ii') was obtained \cite{ErdAI} that relied crucially on this work. 
\subsubsection{The strategy of \cite{ErdAI}}
Our paper is inspired by the method in \cite{ErdAI}, so we take a moment to discuss the main innovations of that paper. The main lemmas to be used in our paper will be stated and proved in the appendix to this paper (the proofs given are our own).\\
To explain its strategy, we recall the following definition.
\begin{def1}
We say that an additive function $g: \mb{N} \ra \mb{R}$ has \emph{finite concentration} if 
$$
\limsup_{x \ra \infty} \frac{1}{x}\sup_{a \in \mb{R}} |\{n \leq x : g(n)\in [a-1/2,a+1/2]\}| > 0.
$$
Erd\H{o}s showed (see \cite[Thm. V]{Erdos}) that this is equivalent to the statement that there exists\footnote{In the sequel we shall sometimes refer to such a parameter $\lambda$ as a \emph{minimiser} for $g$.} $\lambda \in \mb{R}$ with
$$
\sum_p \frac{\min\{(g(p)-\lambda\log p)^2,1\}}{p} < \infty.
$$
\end{def1}
\noindent The property of finite concentration for additive functions is intimately related to the property of ``pretentiousness'' in multiplicative number theory. In particular, if the finite concentration of $g$ (in the second equivalent form above) is witnessed by a parameter $\lambda$, then the multiplicative function $e^{ig(n)}$, which takes values on the unit circle in $\mb{C}$, will be $n^{i\lambda}$-pretentious, i.e.,
$$
\sum_p \frac{1-\text{Re}(e^{ig(p)}p^{-i\lambda})}{p} < \infty.
$$
This makes the property of finite concentration an appropriate stand-in for pretentiousness within the scope of rigidity problems for additive functions.\\ 
The argument in \cite{ErdAI} split into two parts:
\begin{enumerate}[(a)]
\item showing that if $g(n+1) \geq g(n)$ holds on a density $1$ subset of $\mb{N}$ then $g$ has finite concentration;
\item showing that if $g$ has finite concentration and $g(n+1) \geq g(n)$ on a density $1$ set, then $g = c\log$.
\end{enumerate}
The strategy of \cite{ErdAI} was to first normalise and centre $g$, i.e., replace $g(n)$ by $\tilde{g}_X(n) := (g(n) - \lambda_X \log n)/s_X$, so that at the scale $X$, the sum
$$
\sum_{p \leq X} \frac{\min\{\tilde{g}_X(p)^2,1\}}{p} = \frac{1}{s_X^2}\sum_{p \leq X} \frac{\min\{(g(p) - \lambda_X \log p)^2,s_X^2\}}{p}
$$
could be dynamically controlled in terms of $X$. Choosing $\lambda_X$ and $s_X$ so that this expression takes a fixed value as $X$ varies, one enforces a finite concentration property to hold. However, this does not prevent $\tilde{g}_X$ from taking large prime values, one of the obstructions highlighted in \cite{Mangerel}. \\
To deal with this issue, \cite{ErdAI} introduced the idea of ``clipping'' $\tilde{g}_X$, which is to say replacing $\tilde{g}_X$ by a \emph{uniformly bounded} variant $G_X$, that satisfied the same order relations as $g$, i.e., if $g(n+1) \geq g(n)$ then $G_X(n+1) \geq G_X(n)$. Unfortunately, $G_X$ itself is not additive, which must be overcome by \emph{approximating} it by an additive function that is clipped at prime values. Nevertheless, with some care, the monotonicity properties of $G_X$ can be exploited (see the analogous Lemma \ref{lem:UppBd} in this paper for further details). Note that this step required the Matom\"aki-Radziwi\l\l{} type theorem for additive functions from \cite{Mangerel} as a crucial input.
\subsubsection{Proving Theorem \ref{thm:PowBad}}
The proof of Theorem \ref{thm:PowBad} similarly splits into two steps as above. We emphasise that  studying the sets $\mc{L}_f(\omega)$ and $\mc{S}_f(\omega)$ using the ``normalisation and clipping'' scheme just described requires a non-trivial amount of work. For example, in Section \ref{sec:conc} below we prove a new result, Proposition \ref{prop:gapConc}, which allows us to show that if $g$ does not have finite concentration then the event
\begin{equation}\label{eq:close}
\max\{g(n-1),g(n+1)\} \leq g(n) \leq \max\{g(n-1),g(n+1)\} + \Delta
\end{equation}
has logarithmic density $0$, for any $\Delta \geq 0$. We prove this by adapting a method of Ruzsa \cite{Ruzsa}, who studied the clustering of the values $g(n)$, rather than of its differences $g(n) - g(n+1)$, in bounded intervals. A central role is played by a recent result of Tao and Ter\"{a}v\"{a}inen, \cite[Thm. 3.1(ii)]{TT}, building on the breakthrough work of Pilatte \cite{Pil}, which gives a quantitatively strong upper bound for logarithmically-averaged binary correlations of $1$-bounded multiplicative functions. This is applied (after Fourier inversion) to the correlations
$$
\frac{1}{\log x} \sum_{n \leq x} \frac{e^{iug(n)}e^{-iug(n+1)}}{n}, \quad |u| \leq 1,
$$
and a more complicated variant of a ``patching'' argument due to Ruzsa is applied to make these $u$-dependent estimates compatible with one another. See Section \ref{sec:conc} for the details. \\
Assuming that $\delta(\mc{L}_f(\omega)) = 0$, this actually forces the stronger constraint
\begin{equation}\label{eq:nopeaks}
g(n) < \max\{g(n-1),g(n+1)\}
\end{equation}
to hold for all $n$ outside of a set with logarithmic density $0$, to which we apply a variant of the \cite{ErdAI} argument. The main point is that from \eqref{eq:nopeaks} we readily derive in Lemma \ref{lem:valleys} that between the points of the sparse set of exceptions mentioned, $g(n)$ can only have a single local minimum, thereby allowing us to use telescoping arguments based on monotonicity on either side of this minimum.  \\ 
In contrast to the short deduction of \cite{ErdAI} in the finite concentration case, the corresponding argument in this paper is more substantial, and is given in Section \ref{sec:finConc}. In order to prove Theorem \ref{thm:PowBad}, we show that whenever $g$ has finite concentration witnessed by a parameter $\lambda$, then we can construct a positive density subset of $\mc{L}_f(\omega)$ as long as $f_{\lambda}$ has at least two distinct powerfully bad primes. Key to this is the observation that 
by ``$W$-tricking'' the centred function $g_{\lambda}$, we can force it to concentrate very close to its mean. While this is now a natural strategy in the analysis of pretentious multiplicative functions (as e.g. in \cite{Klu} and \cite{KMRig}), in the present context of additive functions it is apparently novel. More precisely,  we show the following. Let $\e > 0$ be small and let $T$ be chosen large enough in terms of $\e$ so that
$$
\sum_{\ss{p > T \\ |g_{\lambda}(p)| > 1}} \frac{1}{p} + \sum_{\ss{p^k > T \\ k \geq 2}} \frac{1}{p^k} < \e.
$$
We set $W$ to be the product of all primes $p \leq T$ with $|g_{\lambda}(p)| > 1$, and consider affine linear forms $Q_1n + a_1$ and $Q_2n + a_2$ such that $(a_j,Q_j) = 1$ and $W|Q_j$ for $j = 1,2$, where $Q_1,Q_2$ are allowed to grow slowly. Then
$$
|g_{\lambda}(Q_1n+a_1) - g_{\lambda}(Q_2n+a_2)|
$$
is bounded uniformly in terms of $Q_1,Q_2$ as above, on a positive density set. This allows us, for example, to show that if $\sup_{q \in \mb{N}} g_{\lambda}(q) = \infty$ (i.e., $g_{\lambda}$ takes unbounded \emph{positive} values) then with a suitable choice of $q$ and $m = Wq n + 1$, we have
$$
|g_{\lambda}(m) - g_{\lambda}(Wqm \pm 1)| \ll 1 
$$
on some positive density subset of $n$. On the other hand, $q$ can be chosen so that
$$
g_{\lambda}(qWm) = g_{\lambda}(qW) + g_{\lambda}(m) \geq g_\lambda(m) + C,
$$
where $C$ is chosen large (in terms of $\Delta$, say), which creates some separation from the neighbouring values $g_\lambda(Wqm\pm 1)$. (Note that adding back $\lambda \log(Wqm)$ on both sides only changes $C$ by $o(1)$ as $m \ra \infty$.) \\
In the case that $\sup_{q \in \mb{N}} g_{\lambda}(q) < \infty$ (and thus $g_{\lambda}(q)$ tends to $-\infty$ on a subsequence of $q$ instead) we prove instead, using a similar but more complicated $W$-tricking strategy, that there exists $C > 0$ such that for any pair of coprime positive integers $q_1,q_2$, 
\begin{equation}\label{eq:minPS}
\min\{|g_{\lambda}(q_1)|,|g_{\lambda}(q_2)|\} \leq C.
\end{equation}
More precisely, assuming the failure of \eqref{eq:minPS}, we produce a positive density set of integers $n$ where $g(n-1)$ and $g(n+1)$ are simultaneously very small, by enforcing divisibility conditions $n \equiv 1 \pmod{q_1}$, $n \equiv -1 \pmod{q_2}$ for $q_1,q_2$ coprime integers for which $g_{\lambda}(q_j)$ is very negative. As it is more technical, we omit a detailed exposition here. In any case, it is easy to show (see Lemma \ref{lem:easyVer}) that the condition \eqref{eq:minPS} implies that $g_{\lambda}$ has at most one powerfully bad prime, and that it is uniformly bounded on positive integers coprime to some $P$ with $\Omega(P) \leq 1$, which is the desired conclusion.
 
\subsubsection{Deducing Theorems \ref{thm:SarkCM} and \ref{thm:multAtMost1}}
We carry out the deductions of Theorems \ref{thm:SarkCM} and \ref{thm:multAtMost1} in Section \ref{sec:apps}. By Theorem \ref{thm:PowBad} there can only be at most one prime $p$ for which $(g_{\lambda}(p^k))_k$ is unbounded. Theorem \ref{thm:SarkCM} follows immediately from this, because if $f$ is completely multiplicative then $g_{\lambda}$ must be completely additive, and the fact that $g_{\lambda}(p^k) = k g_{\lambda}(p)$ is uniformly bounded in $k$, for $p \neq p_0$, forces $g(p) = \lambda \log p$ for all $p \neq p_0$. There remains the issue of showing that $\lambda$ is a non-negative integer whenever $f$ takes positive integer values; this is a straightforward deduction from e.g., the six exponentials theorem from transcendence theory (see \cite[Sec. 2.1]{Lang}). \\
The proof of Theorem \ref{thm:multAtMost1} is a bit more difficult. As a key ingredient we use an old result of K\'{a}tai \cite{Katai78}, which shows that if
\begin{equation}\label{eq:KatLog}
\liminf_{X \ra \infty} \frac{1}{\log X} \sum_{n \leq X} \frac{|g(n+1)-g(n)|}{n} = 0
\end{equation}
then $g = c \log$. To make use of this we use ``clipping'' in a different way to control the influence of the values $g_{\lambda}(p_0^k)$ on the patterns in $\mc{L}_f$, in the potential event that $\mc{B}_{f_{\lambda}} = \{p_0\}$. In this manner we construct a new additive function $h$ for which $h_{\lambda}$ is uniformly bounded on \emph{all} of $\mb{N}$, and such that $\mc{L}_{h} \subseteq \mc{L}_f$ (noting that here, $\mc{L}_f = \mc{L}_g$). This implies that
$$
h(n) \leq \max\{h(n-1),h(n+1)\}
$$
except on a set of logarithmic density zero, and (after a slightly tedious reduction to the case that the set of $p$ with $f(p) \neq 1$ is sufficiently dense, see Lemma \ref{lem:fequal1}) we upgrade this to the claim that
$$
h(n) < \max\{h(n-1),h(n+1)\}
$$
(cf. \eqref{eq:nopeaks}) outside of a log density $0$ set, using a result\footnote{Here, the finite concentration of $h$ is not an obstruction, unlike in the analysis of \eqref{eq:close} above.} of the author \cite{MangerelEq}. Together with the uniform boundedness of $h_{\lambda}$, we again use a telescoping argument to verify that \eqref{eq:KatLog} holds. \\
Consequently, this shows that $h = c'\log$ for some $c' \in \mb{R}$, and yields the desired conclusion since $g(p) = h(p)$, except for at most one prime $p$.
\section{Existence of counterexamples to Conjecture \ref{conj:54}} \label{sec:2}
In this section we will prove Proposition \ref{prop:counter}, which shows that when $\omega > 1$, Conjecture \ref{conj:54} admits counterexamples. 
\begin{proof}[Proof of Proposition \ref{prop:counter}]
Let $A = A(\omega) \geq 2$ be the minimum positive integer such that
$$
\prod_{p^k, \ k \geq 1} \left(1+\frac{1}{p^{kA}}\right) < \omega.
$$
Fix a prime $p_0$. For every prime power $p^k$, $p \neq p_0$, we define
$$
f_j(p^k) = p^{kA} + 1, \quad j = 1,2.
$$
Also, define $f_j(p_0^k) := c_k^{(j)}p_0^{kA}$, where for $k \geq 1$ set
$$
c_k^{(j)} := p_0^{A(-1)^{j-1}\lfloor \log k\rfloor}.
$$ 
Finally, extend $f$ to a multiplicative function. Since $\lfloor \log k\rfloor \leq k$ for all $k \geq 1$, $f$ clearly takes positive integer values. It is obviously true that for each $j = 1,2$ and any $k \geq 0$, $f_j(p) \neq p^k$ for all $p \neq p_0$, since for such primes $p \nmid f_j(p)$. \\
Writing $\nu_{p_0}(m)$ to denote the maximal $k$ such that $p_0^k | m$ and setting $c_0^{(j)} := 1$ for both $j = 1,2$, we find that
\begin{equation}\label{eq:fgenForm}
f_j(n) = n^Ac_{\nu_{p_0}(n)}^{(j)}\prod_{\ss{p^k || n \\ p \neq p_0}} \left(1+\frac{1}{p^{kA}}\right).
\end{equation}
Obviously when $p_0\nmid n$ we have
$$
1 \leq f_j(n)/n^A = \prod_{\ss{p^k || n}} \left(1+\frac{1}{p^{kA}}\right) \leq \prod_{p^k, k \geq 1} \left(1+\frac{1}{p^{kA}}\right) < \omega, 
$$
which verifies (a). Moreover, we have $c_k^{(j)} = f_j(p_0^k)/p_0^{Ak}$, which tends to $\infty$ as $k \ra \infty$ if $j =1$, and tends to $0$ as $k \ra \infty$ if $j = 2$, 
verifying (b). \\
Next, we show that $|\mc{S}_{f_1}(\omega)| < \infty$. Let us choose $n_0(\omega)$ minimally so that for $n \geq n_0(\omega)$, 
$$
\left(1+\frac{1}{n}\right)^{A} \prod_{p^k, \ k \geq 1}\left(1+\frac{1}{p^{kA}}\right) < \omega.
$$
Using \eqref{eq:fgenForm}, we see that
\begin{align*}
\min\{f_1(n - 1), f_1(n+1)\} &= \min_{\e = \pm 1} (n+\e)^A c_{\nu_{p_0}(n+\e)}^{(1)} \prod_{p^k||(n+\e)} \left(1+\frac{1}{p^{Ak}}\right) \\
&\leq n^A \left(1+\frac{1}{n}\right)^A \prod_{p^k, \ k \geq 1} \left(1+\frac{1}{p^{Ak}}\right) \min\{c_{\nu_{p_0}(n-1)}^{(1)}, c_{\nu_{p_0}(n+1)}^{(1)}\} \\
&< \omega f_1(n)\min\left\{\frac{c_{\nu_{p_0}(n-1)}^{(1)}}{c_{\nu_{p_0}(n)}^{(1)}}, \frac{c_{\nu_{p_0}(n+1)}^{(1)}}{c_{\nu_{p_0}(n)}^{(1)}}\right\}.
\end{align*}
Observe that either $p_0|n$, in which case $c_{\nu_{p_0}(n)}^{(1)} \geq 1$ whilst $c_{\nu_{p_0}(n\pm 1)}^{(1)} = c_0^{(1)} = 1$, or else $p_0 \nmid n$, in which case $0 \leq \min\{\nu_{p_0}(n-1),\nu_{p_0}(n+1)\} \leq 1$. Because $c^{(1)}_0 = c^{(1)}_1 = 1$ it follows that the bracketed expression is $\leq 1$, and thus
$$
\min\{f_1(n-1),f_1(n+1)\} < \omega f_1(n)
$$
for all $n \geq n_0(\omega)$. We conclude that $|\mc{S}_{f_1}(\omega)| < \infty$. \\
The proof that $|\mc{L}_{f_2}(\omega)| < \infty$ is similar. In this case, note that $c_{\nu_{p_0}(n)}^{(2)} \leq 1$ for all $n$, and that as $c_0^{(2)} = c_1^{(2)} = 1$, 
$$
\max\left\{\frac{c^{(2)}_{\nu_{p_0}(n-1)}}{c^{(2)}_{\nu_{p_0}(n)}}, \frac{c^{(2)}_{\nu_{p_0}(n+1)}}{c^{(2)}_{\nu_{p_0}(n)}}\right\} \geq 1. 
$$
Let $n_1(\omega)$ be chosen minimally so that if $n \geq n_1(\omega)$ then
$$
\left(1-\frac{1}{n}\right)^{-A} \prod_{p^k,k \geq 1} \left(1+\frac{1}{p^{kA}}\right) < \omega.
$$
Then we deduce from \eqref{eq:fgenForm} that for all $n \geq n_1(\omega)$,
\begin{align*}
\max\{f_2(n-1),f_2(n+1)\} \geq n^A c^{(2)}_{\nu_{p_0}(n)} &\cdot \prod_{\ss{p^k || n \\ p \neq p_0}}\left(1+\frac{1}{p^{kA}}\right) \cdot \left(1-\frac{1}{n}\right)^A \prod_{p^k,k \geq 1} \left(1+\frac{1}{p^{kA}}\right)^{-1} \\
&\cdot \max\left\{\frac{c^{(2)}_{\nu_{p_0}(n-1)}}{c^{(2)}_{\nu_{p_0}(n)}}, \frac{c^{(2)}_{\nu_{p_0}(n+1)}}{c^{(2)}_{\nu_{p_0}(n)}}\right\} > \frac{1}{\omega} f_2(n),
\end{align*}
whence the claim follows.
\end{proof}

\section{Reducing to the case that $g$ has finite concentration} \label{sec:RedtoFinConc}
The main result of this section is the following reduction, inspired by the strategy of \cite{ErdAI}.
\begin{thm}[Reducing to finite concentration] \label{thm:redtoFinConc}
Let $f: \mb{N} \ra (0,\infty)$ be a multiplicative function. Suppose that there is $\omega > 0$ such that $\delta(\mc{L}_f(\omega))\delta(\mc{S}_f(\omega)) = 0$. Set $g:= \log f$. Then $g$ has finite concentration, i.e., there exists $\lambda \in \mb{R}$ such that
$$
\sum_p \frac{\min\{1,(g(p)-\lambda \log p)^2\}}{p} < \infty.
$$
\end{thm}
To prove this theorem we will need a few ingredients. We emphasise here that the results of this section \emph{do not} require that $f(n) \in \mb{N}$, which will only be important later on. \\
For the rest of this section we fix $f$ and $\omega > 0$ and assume that the hypotheses of Theorem \ref{thm:redtoFinConc} hold. In light of Remark \ref{rem:MonoOmega} we may assume that $\omega \geq 1$, and we shall set $\Delta := \log \omega \geq 0$. Also, given \eqref{eq:ftoreciprocal}, by replacing $f$ with $1/f$ if needed we may assume that $\delta(\mc{L}_f(\omega)) = 0$. For convenience, in the sequel given $\lambda \in \mb{R}$ we write $g_{\lambda}(n) := g(n)-\lambda \log n$, and
$$
U_g(x;\lambda) := \sum_{p \leq x} \frac{\min\{1, g_{\lambda}(p)^2\}}{p}, \quad U_g(x) := \min_{|\lambda| \leq x^2} U_g(x;\lambda), \quad V_g(x) := \inf_{\lambda \in \mb{R}} U_g(x;\lambda).
$$
Obviously $U_g(x) \geq V_g(x)$ for all $x$, which we will need later. \\
Given $d \in \mb{N}$ and $h > 0$ we also define the \emph{logarithmic length $d$ gap concentration function} by
$$
\mc{Q}_g(x;d,h) := \frac{1}{\log x} \sup_{a \in \mb{R}} \sum_{\ss{\sqrt{x} < n \leq x \\ |g(n+d)-g(n)-a| \leq h/2}}\frac{1}{n}.
$$
In Section \ref{sec:conc} we will prove the following estimate.
\begin{prop}[A Ruzsa-type concentration estimate for gaps] \label{prop:gapConc}
Let $g$ be a real-valued additive function, let $d \in \mb{N}$ and let $h > 0$. Then
$$
\mc{Q}_g(x;d,h) \ll_{d,h} \frac{1}{\sqrt{U_g(x)}} + \left(\frac{\log\log \log x}{\log\log x}\right)^{1/2}.
$$
\end{prop}
We will also require the following.
\begin{lem}\label{lem:Vfinite}
Assume that $V_g(x) \not\ra \infty$ as $x \ra \infty$. Then $g$ has finite concentration.
\end{lem}
\begin{proof}
In the sequel, write $\lambda(x)$ to denote a minimiser for $V_g$ at scale $x$. Observe first of all that $V_g(x)$ is non-decreasing: if $x' > x$ then by restricting to the primes $p \leq x$, we have
$$
V_g(x') = \sum_{p \leq x'} \frac{\min\{1,(g(p) - \lambda(x')\log p)^2\}}{p} \geq \sum_{p \leq x} \frac{\min\{1,(g(p) - \lambda(x')\log p)^2\}}{p} \geq V_g(x).
$$
Therefore, if $V_g(x) \not\ra \infty$ then there is $B > 0$ such that $V_g(x) \leq B$ for all $x$. \\
Let $T$ be a parameter, chosen sufficiently large that
$$
\sum_{p \leq T} \frac{1}{p} > B.
$$
It follows that for every $x \geq T$,
$$
\sum_{p \leq T} \frac{1}{p} > B \geq V_g(x) \geq V_g(T) = \sum_{p \leq T} \frac{\min\{1,(g(p)-\lambda(x)\log p)^2\}}{p},
$$
and therefore that there is $p \leq T$ such that $|g(p) - \lambda(x) \log p| < 1$. But then we deduce that
$$
|\lambda(x)| \leq 1 + \max_{p \leq T} \frac{|g(p)|}{\log p},
$$
which depends only on $B$ and $g$, but not on $x$. Thus, the map $x \mapsto \lambda(x)$ is uniformly bounded, and by compactness there is a subsequence $(x_k)_k$ and a limit point $\lambda$ such that $\lambda(x_k) \ra \lambda$ as $k \ra \infty$. \\
Let $y \geq 2$. By Fatou's lemma, we have
$$
\sum_{p \leq y} \frac{\min\{1,g_{\lambda}(p)^2\}}{p} = \sum_{p \leq y} \lim_{k \ra \infty} \frac{\min\{1,g_{\lambda(x_k)}(p)^2\}}{p} \leq 
\liminf_{k \ra \infty} V_g(x_k) \leq B.
$$
But since this holds for arbitrary $y$ and that the LHS is non-decreasing in $y$, we see on taking $y \ra \infty$ that
$$
\sum_p \frac{\min\{1,g_{\lambda}(p)^2\}}{p} < \infty,
$$
as claimed.
\end{proof}
For the remainder of this section we shall assume for the sake of contradiction that $g$ does not have finite concentration. By Lemma \ref{lem:Vfinite}, $V_g(x) \ra \infty$.
For $x$ large and $\omega \geq 1$, define
\begin{align*}
\mc{M}_g(\omega) &:= \mc{L}_f(\omega) \cup \{n \in \mb{N} : \max\{g(n+1),g(n-1)\} \leq g(n) \leq \max\{g(n+1),g(n-1)\} + \Delta\}, \\
\mc{M}_g(x;\omega) &:= \mc{M}_g(\omega) \cap [1,x].
\end{align*}
Assuming Proposition \ref{prop:gapConc} for now, we easily obtain the following corollary. 
\begin{cor}[Infinite concentration and clusters of consecutive values] \label{cor:clusters}
Let $\omega \geq 1$, and let $\e > 0$ be small. Assume that $V_g(x) \ra \infty$ as $x \ra \infty$.
Then $\mc{M}_g(\omega)$ has logarithmic density $0$. Consequently, there is an infinite subsequence $(x_k)_k$ on which
$$
|\mc{M}_g(x_k;\omega)| = o_{k \ra \infty}(x_k).
$$
\end{cor}
\begin{proof}
Note that the condition 
\begin{equation}\label{eq:midRange}
\max\{g(n+1),g(n-1)\} \leq g(n) \leq \max\{g(n+1),g(n-1)\} + \Delta
\end{equation}
implies that for one of $\e \in \{-1,+1\}$,
$$
0 \leq g(n)-g(n+\e) \leq \Delta+1.
$$
Put $J := \lfloor \log\log\log x\rfloor$. By the union bound, we deduce that
\begin{align*}
\frac{1}{\log x} \sum_{\ss{n \leq x \\ \eqref{eq:midRange} \text{ holds for $n$}}} \frac{1}{n} &\leq \frac{2}{\log x} \sum_{\ss{n \leq x \\ |g(n+1)-g(n)| \leq \Delta+1}} \frac{1}{n} + o(1) \\
&= 2\sum_{0 \leq j \leq J} \frac{1}{2^j}\left(\frac{1}{\log(x^{2^{-j}})}  \sum_{\ss{x^{2^{-(j+1)}} < n \leq x^{2^{-j}} \\ |g(n+1)-g(n)| \leq \Delta+1}} \frac{1}{n}\right) + o(1) \\
&\leq 2\sum_{0 \leq j \leq J} \frac{1}{2^j} \mc{Q}_g(x^{2^{-j}};1,2(\Delta+1)) + o(1).
\end{align*}
By assumption, we have $V(x) \ra \infty$ as $x \ra \infty$, so that since $U_g(x) \geq V(x)$ for all $x$, Proposition \ref{prop:gapConc} yields 
$$
\frac{1}{\log x} \sum_{\ss{n \leq x \\ \eqref{eq:midRange} \text{ holds for $n$}}} \frac{1}{n} = o(1), \quad x \ra \infty.
$$
Since $\delta(\mc{L}_f(\omega)) = 0$ and $\delta(\mc{A} \cup \mc{B}) \leq \delta(\mc{A}) + \delta(\mc{B})$ whenever $\mc{A}$ and $\mc{B}$ have logarithmic density, this proves the first claim. The second follows from the first by partial summation, since
$$
\frac{1}{\log x} \sum_{\ss{\sqrt{x} < n \leq x \\ n \in \mc{M}_g(\omega)}} \frac{1}{n} \geq \frac{1}{\log x} \int_{\sqrt{x}}^x \left(\frac{|\mc{M}_g(y;\omega)|}{y}\right) \frac{dy}{y} \geq \frac12 \min_{\sqrt{x} \leq y \leq x} \frac{|\mc{M}_g(y;\omega)|}{y},
$$
on selecting, for each $k \geq 1$ the point $x_k$ to be the minimum between points $2^{2^{k-1}}$ and $2^{2^{k}}$. 
\end{proof}
From this corollary, we deduce the following monotonicity data for the function $g$. 
\begin{lem} \label{lem:valleys}
Let $f: \mb{N} \ra (0,\infty)$ be a multiplicative function for which $\delta(\mc{L}_f(\omega))= 0$. Let $(x_k)_k$ be an infinite subsequence, and write 
$$
N_k := |\mc{M}_g(x_k;\omega)|, \quad \mc{M}_g(x_k;\omega) = (m_j)_{1 \leq j \leq N_k}.
$$
Set also $m_0 := 0$. Then for every $k \geq 1$ and $0 \leq j \leq N_k-1$ 
there is $\ell_j \in (m_j,m_{j+1}]$ such that
\begin{enumerate}[(i)]
\item $g(n+1) < g(n) \text{ for all } m_j < n <\ell_j$
\item $g(n) \leq g(n+1) \text{ for all } \ell_j \leq n < m_{j+1}.$
\end{enumerate}
\end{lem}
\begin{proof}
By assumption, for each $n \in \mb{N} \bk \mc{L}_f(\omega)$ we have
$$
g(n)\leq \max\{g(n-1),g(n+1)\} + \Delta.
$$
Therefore, if $m_j < n < m_{j+1}$ for some $j \geq 0$ then we get the stronger upper bound $g(n) < \max\{g(n-1),g(n+1)\}$. Now 
define $\ell_j \in (m_j,m_{j+1}]$ to be the minimal $n$ such that $g(n+1) \geq g(n)$ in the interval (if $g$ is decreasing in the entire interval $(m_j,m_{j+1})$ then set $\ell_j = m_{j+1}$). Now, we argue by induction that $g(n) \leq g(n+1)$ for all $\ell_j \leq n < m_{j+1}$ (the claim being vacuous in the case that $\ell_j = m_{j+1}$). \\
The base case is $g(\ell_j) \leq g(\ell_j+1)$, which holds by construction. Next, assume the same relation holds for each $\ell_j \leq n < m_{j+1}-1$, and let $\ell_j< n < m_{j+1}$. As $g(n) \geq g(n-1)$ and $g(n) < \max\{g(n-1),g(n+1)\}$ it follows that $g(n) < g(n+1)$. The claim then follows by induction.
\end{proof}
We now recall the setup of the clipping and normalisation scheme from \cite{ErdAI}. In addition, we provide additional insight about the mechanism, which is quite general and should have future applications beyond this paper. See also the appendix to this paper for proofs of the key results cited here. \\
Let $K > 1$ and $M \in (0,1/4)$ be parameters to be chosen later. Define the ``clipping'' function $\phi_K : \mb{R} \ra \mb{R}$ by
$$
\phi_K(t) :=  \text{sign}(t) \min\{K, |t|\} = \begin{cases} t &\text{ if } |t| \leq K \\ K &\text{ if } t > K \\ -K &\text{ if } t < -K. \end{cases}
$$
This has a number of nice, easily-checked properties, which we record in the following lemma. We omit the proofs, which follow from simple casework.
\begin{lem}\label{lem:clipProp}
The following properties hold:
\begin{enumerate}[(i)]
\item (Lipschitzness) If $u,v \in \mb{R}$ then $|\phi_K(u+v)-\phi_K(u)| \leq |v|$.
\item (Monotonicity) If $t_1 \leq t_2$ then $\phi_K(t_1) \leq \phi_K(t_2)$.
\item (Contractivity) We have $|\phi_K(t)| \leq \min\{K,|t|\}$ for all $t \in \mb{R}$.
\end{enumerate}
\end{lem}
Let $X$ be a large parameter. Because $g$ does not have finite concentration by assumption, \cite[Lem. 4.1]{ErdAI} (for a proof, see Lemma \ref{lem:4point1} below), we may choose parameters $s_X = s_X(M) > 0$ and $c_X = c_X(M) \in \mb{R}$ so that if $\lambda_X := c_X s_X$ then
\begin{equation}\label{eq:normalise}
M = c_X^2 + \sum_{p \leq X} \frac{\min\{1,(g(p)/s_X - c_X \log p)^2\}}{p} = c_x^2 + \frac{1}{s_X^2}\sum_{p \leq X}\frac{\min\{s_X^2, g_{\lambda_X}(p)^2\}}{p}.
\end{equation}
These parameters satisfy $c_X \ra 0$ and $s_X \ra \infty$, as $X \ra \infty$. By replacing $g$ with $g_{\lambda_X}/s_X$ we have centred and normalised $g$ in a scale-dependent way. Moreover, by fixing $M$, which will be chosen small but fixed (as $X \ra \infty$) in \eqref{eq:normalise}, we have ensured that the set of primes $p$ with $|g_{\lambda_X}(p)| > s_X$ is sparse, and therefore ``typical'' integers $n$ have few such primes with this property. \\
We furthermore define
$$
\mc{T}_X := \{p \leq X : |g_{\lambda_X}(p)| > s_X\}, \quad a_X := \frac{1}{s_X}\sum_{\ss{p \leq X \\ p \notin \mc{T}_X}} \frac{g_{\lambda_X}(p)}{p},
$$
as well as the ``clipped'' function
\begin{align*}
h_X(n) &:= \phi_K(g_{\lambda_X}(n)/s_X - a_X).
\end{align*}
Note that while $h_X$ is \emph{not additive} in general, it is a \emph{uniformly bounded} variant of $g_{\lambda_X}/s_X$ that, by Lemma \ref{lem:clipProp}(i) and (ii), still approximately obeys the same order relations that $g$ itself does (see e.g. \eqref{eq:gtohSign} below).\\
Together with the previous lemmas, we use this latter observation crucially as follows.
\begin{lem}[Pointwise comparison to a short interval average after clipping] \label{lem:MRclipped}
Let $f: \mb{N} \ra (0,\infty)$ be a multiplicative function. Let $\omega \geq 1$ and suppose that $\delta(\mc{L}_f(\omega)) = 0$. Assume that $g = \log f$ satisfies $V_g(x) \ra \infty$ as $x \ra \infty$. Then there is an infinite sequence $(x_k)_k$ and a function $H_k = H(x_k) \ra \infty$ as $k \ra \infty$ for which
$$
\frac{1}{x_k} \sum_{H_k < n \leq x_k} \left|h_{x_k}(n) - \frac{1}{H_k}\sum_{1 \leq j \leq H_k} h_{x_k}(n-j)\right|^2 = o_{k \ra \infty}(1).
$$
\end{lem}
\begin{proof}
By Corollary \ref{cor:clusters}, we obtain a sequence $(x_k)_k$ such that 
$\lim_{k \ra \infty} |\mc{M}_g(x_k;\omega)| /x_k = 0$. For large $k$ let $N_k := |\mc{M}_g(x_k;\omega)|$ and for ease of notation, set $N = N_k$ and $x = x_k$. As in Lemma \ref{lem:valleys}, set $m_0 := 0$ and write $\mc{M}_g(x;\omega) = (m_j)_{1\leq j \leq N}$, so that $m_N \leq x < m_{N+1}$. 
By Lemma \ref{lem:valleys}, we may decompose $[0,x]$ as a union
\begin{equation}\label{eq:split}
[m_{N}, \min\{x+1,\ell_N\}) \cup [\min\{x+1,\ell_N\},x] \cup \bigcup_{0 \leq j \leq N-1} [m_j,\ell_j) \cup [\ell_j,m_{j+1}),
\end{equation}
such that $f$ is decreasing on $(m_j,\ell_j)$ and increasing on $[\ell_j, m_{j+1})$ for each $0 \leq j \leq N-1$, and similarly on the two remaining intervals in \eqref{eq:split}. We note furthermore that using the monotonicity and Lipschitz properties of $\phi_K$ from Lemma \ref{lem:clipProp}, we have that if $g(n+1) \geq g(n)$ then
\begin{align*}
h_{x}(n+1) - h_{x}(n) &= \phi_K(g_{\lambda_{x}}(n+1)/s_x - a_{x}) - \phi_K(g_{\lambda_{x}}(n)/s_x-a_{x}) \\
&= \phi_K(g(n+1)/s_{x} - c_{x}\log n - a_{x}) - \phi_K(g(n)/s_{x} - c_{x}\log n - a_{x}) + O\left(\frac{|c_{x}|}{n}\right),
\end{align*}
wherein the main term on the RHS is non-negative. Similarly, if $g(n) \geq g(n+1)$ then for $n \geq 1$,
\begin{align*}
h_{x}(n) - h_{x}(n+1) &= \phi_K(g_{\lambda_{x}}(n)/s_x - a_{x}) - \phi_K(g_{\lambda_{x}}(n+1)/s_x-a_{x}) \\
&= \phi_K(g(n)/s_{x} - c_{x}\log n - a_{x}) - \phi_K(g(n+1)/s_{x} - c_{x}\log n - a_{x}) + O\left(\frac{|c_{x}|}{n}\right),
\end{align*}
and now the main term on the RHS is also non-negative. We deduce, therefore, that
\begin{align} \label{eq:gtohSign}
|h_{x}(n+1)-h_{x}(n)| = (h_{x}(n+1)-h_{x}(n)) \text{sign}(g(n+1)-g(n)) + O\left(\frac{|c_x|}{n}\right).
\end{align}
Using the partition \eqref{eq:split}, the bound $|h_{x}| \leq K$ and telescoping, we find that
\begin{align*}
&\frac{1}{x} \sum_{n\leq x} |h_{x}(n+1) - h_{x}(n)|  \\
&\leq \frac{1}{x}\sum_{0 \leq j \leq N} \left(\sum_{m_j < n < \min\{x+1,\ell_{j}\}} (h_{x}(n)-h_{x}(n+1)) + \sum_{\min\{x+1,\ell_{j}\} \leq n \leq \min\{x,m_{j+1}-1\}} (h_{x}(n+1) - h_{x}(n))\right) \\
& \qquad\qquad\qquad\qquad \qquad\qquad\qquad\qquad\qquad\qquad\qquad+ \frac{2K(N+1) + O(|c_{x}|\log x)}{x} \\
&= \frac{1}{x} \sum_{0 \leq j \leq N} \left(h_{x}(m_j+1) + h_{x}(\min\{\lfloor x\rfloor + 1, m_{j+1}\}) - 2h_{x}(\min\{\ell_j,\lfloor x\rfloor+1\})\right) + \frac{2K(N+1) + O(|c_{x}|\log x)}{x} \\
&\qquad\qquad\qquad\qquad \qquad\qquad\qquad\qquad\qquad\qquad\qquad\ll \frac{KN + |c_{x}| \log x}{x}.
\end{align*}
We now revert to the notation $x_k$ and $N_k$ in place of $x$ and $N$, to make the dependence on $k$ more explicit. Let $H_k = H(x_k)$ be a parameter to be chosen shortly. 
Using $|h_{x_k}(n)|\leq K$, telescoping and the Cauchy-Schwarz inequality, we have that
\begin{align*}
&\frac{1}{x_k} \sum_{H_k < n \leq x_k} \left|h_{x_k}(n)-\frac{1}{H_k}\sum_{1 \leq j \leq H_k} h_{x_k}(n-j)\right|^2 \\
&= \frac{1}{x_kH_k^2} \sum_{H_k < n \leq x_k} \left|\sum_{1 \leq j \leq H_k} \sum_{0 \leq i \leq j-1} (h_{x_k}(n-i)-h_{x_k}(n-i-1))\right|^2 \\
&= \frac{1}{x_kH_k^2} \sum_{H_k < n \leq x_k} \left|\sum_{0 \leq i \leq H_k-1} (H_k-i)(h_{x_k}(n-i)-h_{x_k}(n-i-1))\right|^2 \\
&\leq \frac{H_k^2}{x_k} \sum_{n \leq x_k} |h_{x_k}(n) - h_{x_k}(n+1)|^2 \ll KH_k^2 \frac{KN_k + |c_{x_k}| \log x_k}{x_k}.
\end{align*}
Since $N_k/x_k \ra 0$, and $|c_{x_k}| \leq 1$ as $k \ra \infty$, we may select
$$
H_k := \left\lfloor x_k^{1/3}(KN_k + \log x_k)^{-1/3}\right\rfloor,
$$
and then
\begin{align*}
\frac{1}{x_k} \sum_{H < n \leq x_k} \left|h_{x_k}(n)-\frac{1}{H_k}\sum_{1 \leq j \leq H_k} h_{x_k}(n-j)\right|^2 \ll K H_k^{-1} = o_{k \ra \infty}(1),
\end{align*}
as claimed.
\end{proof}
To finish the proof of Theorem \ref{thm:redtoFinConc} we will incorporate Lemma \ref{lem:MRclipped} into the proof scheme of \cite{ErdAI}. For this we recall some further notation. Define
$$
\tilde{g}_x(n) := \sum_{p|n} \phi_K(g_{\lambda_x}(p)/s_x), \quad \mu(x) := \sum_{p \leq x} \frac{\tilde{g}_x(p)}{p}, \quad B_{\tilde{g}_x}(x)^2 := \sum_{p \leq x} \frac{\tilde{g}_x(p)^2}{p}.
$$
We note that $\tilde{g}_x$ is a strongly additive function, i.e., $\tilde{g}_x(p^\nu) = \tilde{g}_x(p)$ for all primes $p$ and $\nu \geq 1$, and $\mu(x)$ is (asymptotically) its average value over $[1,x]$. Moreover, $\tilde{g}_x$ has ``clipped'' prime values, so that $|\tilde{g}_x(p)| \leq K$, and because of \eqref{eq:normalise}, 
\begin{equation}\label{eq:Btildeg}
B_{\tilde{g}_x}(x)^2 \leq K^2M.
\end{equation}
A key idea from \cite{ErdAI} (see Lemma 5.1 there, or Lemma \ref{lem:5point1} below) is that one can control the variance between $h_x$ and $\tilde{g}_x - a_x$ explicitly in terms of $K$ and $M$. This is used in the proof of the following lemma.
\begin{lem} \label{lem:UppBd}
Let $x$ be large and let $10 \leq H \leq \sqrt{x}$ be an integer-valued function of $x$ with $H = H(x) \ra \infty$ as $x \ra \infty$. Let $M$ and $K$ be as above, and assume that $x \geq x_0(M,K)$. Then there is an absolute constant $C>0$ such that
\begin{align*}
\frac{1}{x}\sum_{n \leq x}\left|\tilde{g}_x(n) - \mu(x)\right|^2 \leq C\left(\frac{M}{K^2} + M^2\log(1/M) + (KM)^2\right) &+ \frac{4}{x}\sum_{n \leq x} \left|h_x(n)-\frac{1}{H}\sum_{1\leq j \leq H} h_x(n-j)\right|^2 \\
&+ o_{x \ra \infty}(1).
\end{align*}
\end{lem}
\begin{proof}
Following \cite{ErdAI}, for an arithmetic function $F$ write 
$$
\mc{A}_HF(n) := \frac{1}{H}\sum_{1 \leq j \leq H} F(n-j), \quad n \in \mb{N}.
$$ 
Set also $m_x := \mu(x) - a_x$ and $G_x(n) := \tilde{g}_x(n) - a_x$. We then have the decomposition
\begin{align*}
\tilde{g}_x - \mu(x) = 
G_x-m_x=(G_x-h_x)+(h_x-\mathcal A_H h_x)
            +\mathcal A_H(h_x-G_x)+(\mathcal A_H G_x-m_x).
\end{align*}
By H\"{o}lder's inequality, 
$$
\frac{1}{x} \sum_{H < n \leq x} |\mc{A}_H(h_x -G_x)(n)|^2 \leq \frac{1}{x}\sum_{n \leq x} |h_x(n)-G_x(n)|^2.
$$
Next, \cite[Lem. 5.1]{ErdAI} gives that there exists an absolute constant $C' > 0$ such that
\begin{align*}
\frac{1}{x}\sum_{n\leq x} |h_x(n)-G_x(n)|^2 &\leq C'\left(\frac{M}{K^2} + M^2\log(1/M) + (KM)^2\right) + o_{x \ra \infty}(1)
\end{align*}
(for a proof, see Lemma \ref{lem:5point1}). Finally, as $|\tilde{g}_x(p)| \leq K$ and $H = H(x) \ra \infty$ with $x$, \cite[Lem. 2.1]{ErdAI} yields\footnote{This result relies on the fact that $\tilde{g}_x$ is additive, and therefore the Matom\"{a}ki-Radziwi\l\l{} theorem for additive functions, \cite[Thm. 1.1]{Mangerel} applies.}
\begin{align*}
\frac{1}{x}\sum_{H < n \leq x} |\mc{A}_H G_x(n) - m_x|^2 = \frac{1}{x}\sum_{H < n \leq x} |\mc{A}_H \tilde{g}_x(n) - \mu(x)|^2 = o_{x \ra \infty}(1)
\end{align*}
(for a proof, see Lemma \ref{lem:2point1} below).
Combining these with Cauchy-Schwarz, 
we find that there is a constant $C = 8C' > 0$ such that
\begin{align*}
&\frac{1}{x}\sum_{H < n\leq x} |\tilde{g}_x(n)-\mu(x)|^2 \\
&\leq \frac{8}{x}\sum_{n \leq x} |h_x(n)-G_x(n)|^2 + \frac{4}{x}\sum_{H < n \leq x} |\mc{A}_H G_x(n)-m(x)|^2 + \frac{4}{x} \sum_{H < n\leq x} |h_x(n) - \mc{A}_Hh_x(n)|^2 \\
&\leq C\left(\frac{M}{K^2}  + (MK)^2\right) + \frac{4}{x} \sum_{H < n \leq x} |h_x(n)-\mc{A}_H h_x(n)|^2 + o_{x \ra\infty}(1).
\end{align*}
Since $|\tilde{g}_x(n) - \mu(x)| \ll K \log x$ for all $n \leq x$, the terms $n \leq H$ contribute $O(H(K\log x)^2/x)$ to this sum. As $H \leq\sqrt{x}$, the claim follows.
\end{proof}
In the other direction, the following lemma was derived in \cite{ErdAI} using a variance estimate for additive functions, due to Ruzsa \cite{RuzsaVar}. For a proof, see Lemma \ref{lem:LowBdApp} below.
\begin{lem}[Lem. 5.2 of \cite{ErdAI}]\label{lem:LowBd}
There are absolute constants $c_0,C_0 > 0$ such that
$$
\liminf_{X \ra \infty} \frac{1}{X}\sum_{n \leq X} |\tilde{g}_X(n) - \mu(X)|^2 \geq c_0M - C_0(KM)^2.
$$
\end{lem}
\begin{proof}[Proof of Theorem \ref{thm:redtoFinConc}]
Assume that $f: \mb{N} \ra (0,\infty)$ is multiplicative and satisfies $\delta(\mc{L}_f(\omega)) \delta(\mc{S}_f(\omega)) = 0$, for some $\omega \geq 1$. In light of \eqref{eq:ftoreciprocal} we have $\mc{S}_f(\omega) = \mc{L}_{1/f}(\omega)$, so that on replacing $f$ by $1/f$ if needed we may assume that $\delta(\mc{L}_f(\omega)) = 0$. \\
If $g = \log f$ does not have finite concentration then by Lemma \ref{lem:Vfinite}, $V_g(x) \ra \infty$ as $x \ra \infty$. On combining Lemmas \ref{lem:MRclipped} and \ref{lem:UppBd}, we obtain 
$$
\frac{1}{x_k} \sum_{n \leq x_k} |\tilde{g}_{x_k}(n)-\mu(x_k)|^2  \leq C_1 \left(\frac{M}{K^2} + M^2 \log(1/M) +  K^2M^2\right) + o_{k \ra \infty}(1).
$$
On the other hand, by Lemma \ref{lem:LowBd} we obtain the lower bound
$$
\frac{1}{x_k} \sum_{n \leq x_k} |\tilde{g}_{x_k}(n) - \mu(x_k)|^2 \geq c_0 M - C_0 K^2M^2 - o_{k \ra \infty}(1).
$$
Finally, choose $K > 1$ large enough that $C_1/K^2 < c_0/4$, then $M$ small enough that
$$
M \log(1/M) + (C_1 + C_0)K^2M < c_0/4.
$$
On rearranging, we arrive at the estimate
$$
\frac{c_0M}{2} = o_{k \ra \infty}(1),
$$
which is a contradiction as soon as $k$ is sufficiently large. The proof of Theorem \ref{thm:redtoFinConc} is now complete.
\end{proof}
\section{The concentration of the shifts $g(n+1)-g(n)$} \label{sec:conc}
In this section we will prove Proposition \ref{prop:gapConc}.
Let $g$ be a real-valued additive function, let $d \in \mb{N}$ and let $h > 0$. For convenience, given $u \in \mb{R}$ let us define\footnote{We use the standard notation $e(t) := e^{2\pi i t}$, for $t \in \mb{R}$.} the multiplicative function $f_u(n) := e(ug(n))$ and 
$$
M_X(f_u;T,Q) := \inf_{|t| \leq T} \min_{\ss{q \leq Q \\ \chi \pmod{q} \text{ primitive}}} \mb{D}(f_u, \chi(n)n^{it}; X)^2,
$$
where for arithmetic functions $f_1,f_2$ bounded in absolute value by $1$, the \emph{pretentious distance} between $f_1$ and $f_2$ up to $X$ is defined by
$$
\mb{D}(f_1,f_2;X)^2 := \sum_{p \leq X} \frac{1-\text{Re}(f_1(p)\bar{f}_2(p))}{p}.
$$
Since we will need this later, we recall the following properties of $\mb{D}$, see e.g. \cite[Lem. 3.1]{GSPret}. We will refer to any of these as instances of the \emph{triangle inequality} for the pretentious distance:
\begin{align*}
\mb{D}(f_1,f_3;X) &\leq \mb{D}(f_1,f_2;X) + \mb{D}(f_2,f_3;X) \\
\mb{D}(f_1f_2,f_3f_4;X) &\leq \mb{D}(f_1,f_3;X) + \mb{D}(f_2,f_4;X).
\end{align*}
In the sequel we set $Q = Q(X) := (\log X)^{1/125}$. 
All estimates will be allowed to depend on $d$ and $h$, and we will keep this dependence implicit.
\begin{lem}
Let $X \leq T \leq X^2$. Then there is an absolute constant $c > 0$ such that 
$$
\mc{Q}_g(X;d,h) \ll \int_{-1}^{1} e^{-cM_X(f_{u}; T, Q)} du.
$$
\end{lem}
\begin{proof}
By replacing $h$ by $\lceil h\rceil$ (which only increases $\mc{Q}_g(X;d,h)$), we may assume that $h \in \mb{N}$. Next, splitting the interval $[a-h/2,a+h/2]$ into unit intervals and using the pigeonhole principle, we have
\begin{align}
\mc{Q}_g(X;d,h) &= \frac{1}{\log X} \sup_{a \in \mb{R}} \sum_{\ss{\sqrt{X} < n \leq X \\ |g(n+d)-g(n)-a| \leq h/2}} \frac{1}{n} \ll \frac{h+1}{\log X} \sup_{a \in \mb{R}} \max_{|j| \leq h} \sum_{\ss{\sqrt{X} < n \leq X \\ |g(n+d)-g(n) - (a+j)| \leq 1/2}} \ll_h \mc{Q}_g(X;d,1), \label{eq:passto1} 
\end{align}
so that in the sequel we may in fact assume that $h = 1$. \\
Let us define 
$$
\mc{I} := \frac{1}{\log X} \int_{-1}^{1} \left(\int_{\sqrt{X}}^X \left|\frac{1}{x}\sum_{x < n \leq 2x} f_u(n+d)\bar{f}_u(n)\right| \frac{dx}{x}\right) du.
$$
First, using the inequality 
$$
1_{[-1/2,1/2]}(\alpha) \leq \frac{\pi^2}{4} \left(\frac{\sin (\pi\alpha)}{\pi \alpha}\right)^2 = \frac{\pi^2}{4}\int_{-1}^1 (1-|u|) e(\alpha u) du,
$$
we obtain 
\begin{align*}
&\frac{1}{\log X} \int_{\sqrt{X}}^X |\{x < n \leq 2x : \ |g(n+d)-g(n) - a| \leq 1/2\} \frac{dx}{x^2} \\
&\leq \frac{\pi^2}{4\log X} \int_{\sqrt{X}}^{X} \int_{-1}^{1} e(-u a) (1-|u|) \left(\sum_{x < n \leq 2x} f_{u}(n+d)\bar{f}_{u}(n)\right) du \frac{dx}{x^2}.
\end{align*}
Taking absolute values and applying the triangle inequality, we see that
$$
\sup_{a \in \mb{R}} \frac{1}{\log X} \int_{\sqrt{X}}^X |\{x < n \leq 2x : \ |g(n+d)-g(n) - a| \leq 1/2\} \frac{dx}{x^2}  \ll \mc{I}.
$$
Furthermore, for each $a$ the LHS above is
$$
\frac{1}{\log X} \sum_{\sqrt{X} < n \leq 2X} 1_{|g(n+d)-g(n)-a| \le 1/2} \int_{n/2}^n \frac{dx}{x^2} \gg \frac{1}{\log X} \sum_{\sqrt{X} < n \leq 2X} \frac{1_{|g(n+d)-g(n) - a| \leq 1/2}}{n}.
$$
It follows from this and \eqref{eq:passto1} that
$$
\mc{Q}_{g}(X;d,h) \ll \mc{I}.
$$
Finally, by\footnote{Strictly speaking, in \cite{TT} the claimed bound holds with $T_0 = X^2$. However, in the non-pretentious case (ii) being considered here, the result in question is derived from \cite[Thm. 1.7]{MRT}, which uses the threshold $T = X$ (and therefore also any $T \geq X$) instead. We will use this added flexibility in the choice of $T$ to our advantage in the sequel.} \cite[Thm. 3.1(ii)]{TT}, in particular (3.9) there, there is a constant $c > 0$ such that for each $u \in [-1,1]$,
$$
\frac{1}{\log X} \int_{\sqrt{X}}^{X}  \left|\sum_{n \sim x} f_{u}(n+d)\bar{f}_{u}(n)\right|   \frac{dx}{x^2}\ll \mc{L}^{-c},
$$
whenever $\log \mc{L} \leq M_X(f_{u};T,Q)$ with $X \leq T \leq X^2$. Of course, we can simply select $\mc{L} = e^{M_X(f_{u};T,Q)}$, which yields
\begin{align} \label{eq:prelim}
\mc{Q}_g(X;d,h) \ll \int_{-1}^1 e^{-cM_X(f_{u};T_0,Q)} du,
\end{align}
as claimed.
\end{proof}
Write $\lambda(S)$ to denote the Lebesgue measure of $S \subseteq \mb{R}$. Let $K,R \geq 1$ be parameters to be chosen later, subject to $1 \leq R \leq (\log\log X)^{1/2}$ and $1 \leq K \leq \log\log\log X$, say. As above, let $X \leq T \leq X^2$. Let us also define $\tilde{Q} := Q^{20R}$, and note trivially that $M_X(f_u; T,Q) \geq M_X(f_u;X^2,\tilde{Q})$ for all $u$, so that the previous lemma also yields
$$
\mc{Q}_g(X;d,h) \ll \int_{-1}^{1} e^{-cM_X(f_u;X^2,\tilde{Q})} du.
$$
In the sequel, for each $u \in [-1,1]$ we select a primitive character $\chi_u$ with conductor $q_u \leq Q$ and $t_u \in [-X,X]$ (breaking ties by selecting $\chi_u$ of maximal conductor and $t_u$ maximally) such that (on taking $T = X$)
$$
M_X(f_u;X,Q) = \mb{D}(f_u,\chi_u n^{it_u};X)^2.
$$
We additionally select a primitive character $\tilde{\chi}_u$ with conductor $\tilde{q}_u \leq \tilde{Q}$ and $\tilde{t}_u \in [-X^2,X^2]$ such that (on taking $T = X^2$)
$$
M_X(f_u;X^2,\tilde{Q}) = \mb{D}(f_u,\tilde{\chi}_u n^{i\tilde{t}_u};X)^2.
$$
Finally, given a Dirichlet character $\psi$, write $\psi^\ast$ to denote the primitive character that induces it, noting that if $q^\ast$ and $q$ are the conductor and modulus of $\psi$, respectively, then for any $1$-bounded function $G$,
\begin{equation}\label{eq:toImprim}
\mb{D}(G,\psi^\ast;X)^2 = \mb{D}(G,\psi;X)^2 + O\left(\sum_{p|q/q^\ast} \frac{1}{p}\right) = \mb{D}(G,\psi;X)^2 + O(\log\log\log(3q)).
\end{equation}
We will employ a method due to Ruzsa \cite{Ruzsa} (elaborating on previous work of Hal\'{a}sz \cite{Halasz}) to study the dependence on $u$ of the parameters $t_u$ and $\chi_u$. The treatment of $\chi_u$ is novel in this application. \\
For each $k \geq 1$ we consider the set
$$
Y_k := \{u \in [-1,1] : \ \mb{D}(f_u,\chi_u n^{it_u};X) \leq k\}.
$$
Since $Y_k \subseteq Y_{k+1}$ and $\bigcup_{k \geq 1} Y_k = [-1,1]$, it follows that there is a minimal $k_0$ such that $\lambda(Y_k) \geq 1/R$ for all $k \geq k_0$. By the previous lemma, provided $K \geq k_0$ we therefore have
\begin{align} \label{eq:tildeQgBd}
\mc{Q}_g(X;d,h) \ll \sum_{k_0 \leq k \leq K} e^{-ck^2} \lambda(Y_k) + e^{-cK^2} + \frac{1}{R}.
\end{align}
We will need an upper bound for $\lambda(Y_k)$. In the sequel, given $k_0 \leq k \leq K$ we define
\begin{equation}\label{eq:rkEst}
r_k := \lceil 12/\lambda(Y_k)\rceil,
\end{equation}
noting that $r_k \leq 20 R$ for $k \geq k_0$. Write $A+B := \{a+b : a \in A, \ b \in B\}$, and inductively $mA := (m-1)A + A$. By \cite[Lem. 5.3]{Ruzsa}, we have $[-1,1] \subseteq r_k Y_k$.
It follows that for every $u \in [-1,1]$ there exist $u_1,\ldots,u_{r_k} \in Y_k$ such that
\begin{equation*}
u = u_1 + \cdots + u_{r_k}.
\end{equation*}
Therefore, given $u \in [-1,1]$ we may define
\begin{equation}\label{eq:mfqkDef}
\mf{q}_k(u) := \min\{\tilde{q}_u \prod_{1 \leq j \leq r_k} q_{u_j} : \ u = u_1 + \cdots + u_{r_k}, \ u_j \in Y_k \text{ for all } 1 \leq j \leq r_k\}.
\end{equation}
Noting that $\mf{q}_k(u) \leq \tilde{Q}$ for all $u \in [-1,1]$, we may thus define
$$
\mf{q}_k := \max\{\mf{q}_k(u) : \ u \in [-1,1]\}.
$$

\begin{lem}\label{lem:UppBdMX}
For every $k \geq k_0$ we have
$$
\max_{u \in [-1,1]} M_X(f_u; X^2,\tilde{Q}) \leq (r_k k)^2 + O(\log\log\log (10\mf{q}_k)).
$$
\end{lem}
\begin{proof}
Let $u \in [-1,1]$. By the above remarks, we select a representation
\begin{equation}\label{eq:repYk}
u = u_1 + \cdots + u_{r_k}, \quad u_j \in Y_k
\end{equation}
that achieves the minimum in \eqref{eq:mfqkDef}. Since $f_u = f_{u_1} \cdots f_{u_{r_k}}$, 
it follows by the pretentious triangle inequality that
\begin{align}\label{eq:passtoreps}
\mb{D}(f_u, \chi_{u_1} \cdots \chi_{u_{r_k}} n^{i(t_{u_1} + \cdots + t_{u_k})};X) \leq \sum_{1 \leq j \leq r_k} \mb{D}(f_{u_j},\chi_{u_j}n^{it_{u_j}};X) \leq r_k k.
\end{align}
As $R \leq \sqrt{X}$, for large enough $X$ we have
$$
|t_{u_1} + \cdots + t_{u_{r_k}}| \leq R X < X^2,
$$
and also that $\chi_{u_1}\cdots \chi_{u_{r_k}}$ has conductor bounded above by 
$$
q_{u_1}\cdots q_{u_{r_k}} \leq Q^{r_k} \leq Q^{20R} = \tilde{Q}.
$$ 
By minimality and \eqref{eq:toImprim}, we have
\begin{align}\label{eq:minBd}
M_X(f_u;X^2,\tilde{Q}) 
&\leq \mb{D}(f_u,(\chi_{u_1}\cdots \chi_{u_{r_k}})^\ast n^{i(t_{u_1} + \cdots + t_{u_{r_k}})};X)^2 \nonumber\\
&\leq \mb{D}(f_u,\chi_{u_1}\cdots \chi_{u_{r_k}} n^{i(t_{u_1} + \cdots + t_{u_{r_k}})};X)^2 + O(\log\log\log(10\mf{q}_k)) \nonumber\\
&\leq (r_k k)^2 + O(\log\log\log(10\mf{q}_k)).
\end{align}
Since this holds for \emph{all} $u \in [-1,1]$, the claim follows.
\end{proof}
\begin{lem}\label{lem:Cauchy}
Assume that $RK < \tfrac{1}{400}\sqrt{\log\log X}$. Then there is a constant $C > 0$ such that if $u_1,u_2, u_1 + u_2 \in [-1,1]$ then
$$
\tilde{\chi}_{u_1+u_2} = (\tilde{\chi}_{u_1}\tilde{\chi}_{u_2})^{\ast}, \quad |\tilde{t}_{u_1+u_2} - \tilde{t}_{u_1} - \tilde{t}_{u_2}| \ll \frac{e^{9(r_kk)^2}(\log\log(10\mf{q}_k))^{C}}{\log X}.
$$
\end{lem}
\begin{proof}
By Lemma \ref{lem:UppBdMX} and the Cauchy-Schwarz inequality, we obtain that whenever $u$ has a representation \eqref{eq:repYk}, 
\begin{align} \label{eq:compchiu}
\mb{D}(\tilde{\chi}_{u} n^{i\tilde{t}_u}, \chi_{u_1} \cdots \chi_{u_{r_k}} n^{i(t_{u_1} + \cdots + t_{u_k})};X)^2 &\leq 2\left(\mb{D}(f_u,\tilde{\chi}_u n^{i\tilde{t}_u};X)^2 + r_k\sum_{1 \leq j \leq r_k} \mb{D}(f_{u_j},\chi_{u_j} n^{it_{u_j}};X)^2\right)  \nonumber\\
&\leq 4(r_k k)^2 + O(\log\log\log(10\mf{q}_k)).
\end{align}
Now we recall from \cite[Lem. 7.8]{KMTPLMS} that if $\tilde{\chi}_u$ is not induced by $\chi_{u_1} \cdots \chi_{u_{r_k}}$ then
$$
\inf_{|t| \leq 3X^2} \mb{D}(\bar{\tilde{\chi}}_u \chi_{u_1}\cdots \chi_{u_{r_k}}, n^{it};X)^2 \geq \frac{1}{4}\log\left(\frac{\log X}{\log q}\right) + O(1).
$$
where $q$ is the conductor of the character $\bar{\tilde{\chi}}_u \chi_{u_1} \cdots \chi_{u_{r_k}}$. As before,
$$
q \leq \tilde{q}_u \prod_{1 \leq j \leq r_k} q_{u_j}\leq \mf{q}_k \leq Q^{40R}.
$$
Since $Q = (\log X)^{1/125}$ and $R \leq (\log\log X)^{1/2}$, we deduce in this case that
\begin{equation}\label{eq:repulsion}
\inf_{|t| \leq 3X^2} \mb{D}(\bar{\tilde{\chi}}_u \chi_{u_1}\cdots \chi_{u_{r_k}}, n^{it};X)^2 \geq \left(\frac{1}{4} - o(1)\right) \log\log X.
\end{equation}
Otherwise, if $\tilde{\chi}_u$ is induced by $\chi_{u_1}\cdots \chi_{u_{r_k}}$ then either $|t| \geq 1$ and the same bound \eqref{eq:repulsion} holds, or else $0 \leq |t| \leq 1$ and
\begin{equation}\label{eq:nearby}
\mb{D}(\bar{\tilde{\chi}}_u \chi_{u_1}\cdots \chi_{u_{r_k}}, n^{it};X)^2 = \log(1+|t|\log X) + O(\log\log(3\omega(q))) = \log(1+|t|\log X) + O(\log\log\log(10\mf{q}_k)),
\end{equation}
(using the trivial bound $\omega(q) \leq \log q$).  Since for $k_0 \leq k \leq K$,
$$
(r_k k)^2 \leq (20RK)^2 < \frac{1}{400}\log\log X,
$$
we conclude from \eqref{eq:compchiu} and \eqref{eq:nearby} that for some $C > 0$,
$$
\tilde{\chi}_u = (\chi_{u_1}\cdots \chi_{u_{r_k}})^\ast, \quad \tilde{t}_u = t_{u_1} + \cdots + t_{u_{r_k}} + O\left(\frac{e^{4(r_kk)^2}(\log\log(10\mf{q}_k)^{C}}{\log X}\right).
$$
Whenever $u_1,u_2,u_1+u_2 \in [-1,1]$, we obtain again Lemma \ref{lem:UppBdMX} that 
\begin{align*}
&\mb{D}(\tilde{\chi}_{u_1 + u_2} \bar{\tilde{\chi}}_{u_1} \bar{\tilde{\chi}}_{u_2}, n^{i(\tilde{t}_{u_1 + u_2} - \tilde{t}_{u_1} - \tilde{t}_{u_2})}; X)^2 \\
&\leq 3\left(\mb{D}(f_{u_1+ u_2}, \tilde{\chi}_{u_1+u_2} n^{i\tilde{t}_{u_1+u_2}};X)^2 + \mb{D}(f_{u_1},\tilde{\chi}_{u_1} n^{i\tilde{t}_{u_1}};X)^2 + \mb{D}(f_{u_2}, \tilde{\chi}_{u_2} n^{i\tilde{t}_{u_2}};X)^2\right) \\
&\leq 9(r_k k)^2 + O(\log\log\log(10\mf{q}_k)),
\end{align*}
and arguing as just above (increasing $C$ as needed), since $(r_kk)^2 \leq \frac{1}{400} \log\log X$ we find
$$
\tilde{\chi}_{u_1+u_2} = (\tilde{\chi}_{u_1} \tilde{\chi}_{u_2})^{\ast}, \quad \tilde{t}_{u_1+u_2} = \tilde{t}_{u_1} + \tilde{t}_{u_2} + O\left(\frac{e^{9(r_k k)^2}(\log\log(10\mf{q}_k))^{C}}{\log X}\right),
$$
as claimed.
\end{proof}
\begin{lem} \label{lem:complambda}
For each $u \in [-1,1]$, $\tilde{\chi}_u$ is trivial. Moreover, there is a $\lambda \in [-X^2,X^2]$ such that for each $u \in [-1,1]$,
$$
\tilde{t}_u = \lambda u + O\left(\frac{e^{9(r_kk)^2}}{\log X}\right).
$$
In particular, whenever $k_0 \leq k \leq K$ we have
\begin{equation}\label{eq:complambda}
\mb{D}(f_u, n^{i\lambda u}; X)^2 \ll (r_k k)^2.
\end{equation}
\end{lem}
\begin{proof}
Applying the previous lemma inductively, we find that for each $m \geq 1$, whenever $u_1,\ldots,u_m, u_1+ \cdots +u_m \in [-1,1]$ we have
\begin{equation}\label{eq:CauchyChar}
\tilde{\chi}_{u_1+\cdots+u_m} = (\tilde{\chi}_{u_1} \cdots \tilde{\chi}_{u_m})^{\ast}.
\end{equation}
Since there are at most finitely many primitive characters with conductor at most $\tilde{Q}$, by the pigeonhole principle there is a primitive character $\psi$ with conductor at most $\tilde{Q}$ such that 
$$
\mc{S} := \{u \in [-1,1] : \tilde{\chi}_u = \psi\}
$$
has positive outer Lebesgue measure, and it follows that whenever $v \in (\mc{S}-\mc{S}) \cap [-1,1]$, $\tilde{\chi}_v$ is principal and primitive, thus trivial. \\
Next, we check that $\mc{S}$ is in fact measurable, so that its Lebesgue measure, which is equal to its outer measure, is also positive. We note first of all that since $K \leq \log\log\log X$, whenever $k_0 \leq k \leq K$, \cite[Lem. 7.8]{KMTPLMS} shows that if $\tilde{\chi}_u$ and $\tilde{\chi}_u'$ were distinct primitive characters with
$$
M_X(f_u;X^2,\tilde{Q}) = \mb{D}(f_u,\tilde{\chi}_un^{it_u};X)^2 = \mb{D}(f_u,\tilde{\chi}'_u n^{it'_u};X)^2 \leq k^2
$$
then by the triangle inequality,
$$
\left(\frac{1}{4}-o(1)\right) \log\log X \leq \mb{D}(\tilde{\chi}_u \bar{\tilde{\chi}'_u}, n^{i(t_u-t'_u)};X)^2 \leq 4k^2 = o(\log\log X),
$$
which is a contradiction for large $X$. Thus, $\tilde{\chi}_u$ is uniquely determined for each $u$. Now, write $\Xi_{\tilde{Q}}$ to denote the set of primitive characters of conductor at most $\tilde{Q}$, and for each $\psi' \in \Xi_{\tilde{Q}}$ set 
$$
m_{\psi'}(u) := \min_{|t| \leq X^2} \mb{D}(f_u,\psi' n^{it};X)^2.
$$
Then by definition, 
\begin{align*}
\mc{S} = \{u \in [-1,1] : \ m_{\psi}(u) = \min_{\psi' \in \Xi_{\tilde{Q}}} m_{\psi'}(u)\},
\end{align*}
and since the mapping $u \mapsto m_{\psi'}(u)$ is continuous in $u$ for each $\psi'$, $\mc{S}$ is the level set of the minimum of a finite set of continuous functions, therefore measurable. \\ 
Having shown $\mc{S}$ is measurable, we deduce that $\mc{S}$ has positive Lebesgue measure. By Steinhaus' lemma \cite{Steinhaus}, there is a non-empty open subinterval $I \subseteq (\mc{S}-\mc{S}) \cap [-1,1]$, and $\tilde{\chi}_v$ is trivial for all $v \in I$. Taking iterated sumsets, we can find $m \in \mb{N}$ such that $mI$ contains $[-1,1]$. But every element $v \in mI$ is of the shape $v = v_1 + \cdots v_m$ with $\tilde{\chi}_{v_j}$ principal, and therefore by \eqref{eq:CauchyChar}, $\tilde{\chi}_v$
is trivial as well. This proves the first claim. \\
Next, we show that $\mf{q}_k = 1$ for each $k_0 \leq k \leq K$. Let $u \in Y_k$. Using the conclusion that $\tilde{\chi}_u$ is trivial and applying the triangle inequality and Lemma \ref{lem:UppBdMX}, we have
\begin{align*}
\mb{D}(\chi_u, n^{i(t_u - \tilde{t}_u)};X)^2 &\leq 2(\mb{D}(f_u,\chi_u n^{it_u};X)^2 + \mb{D}(f_u,n^{i\tilde{t}_u};X)^2) \leq 4(r_kk^2) + O(\log\log\log(10\mf{q}_k)) \\
&< \left(\frac{1}{100} + o(1)\right) \log\log X.
\end{align*}
On the other hand, another application of \cite[Lem. 7.8]{KMTPLMS} gives that if $\chi_u$ were non-trivial for some $u \in Y_k$ then as $|t_u - \tilde{t}_u| \leq 2X^2$ and $q_u \leq Q \leq \log X$,
$$
\mb{D}(\chi_u,n^{i(t_u-\tilde{t}_u)};X)^2 \geq \left(\frac{1}{4} + o(1)\right)\log \log X,
$$
which is a contradiction. Therefore, $\chi_u$ is trivial for all $u \in Y_k$ as well. By \eqref{eq:mfqkDef}, $\mf{q}_k(u) = 1$ for all $u \in [-1,1]$, and therefore $\mf{q}_k = 1$.\\
We may thus improve the $t_u$ aspect of Lemma \ref{lem:Cauchy} to
$$
\max_{\ss{u_1,u_2 \in [-1,1]: \\ u_1 + u_2 \in [-1,1]}} |\tilde{t}_{u_1 + u_2} - \tilde{t}_{u_1} - \tilde{t}_{u_2}| \ll \frac{e^{9(r_kk)^2}}{\log X}.
$$
Applying Lemma 5.11 of \cite{Ruzsa} with $K$ (in the notation there) given by $e^{9(r_kk)^2}/\log X$, we deduce that there is $\lambda = \tilde{t}_1 \in [-X^2,X^2]$ such that 
$$
\max_{u \in [-1,1]} |\tilde{t}_u - \lambda u| = O\left(\frac{e^{9(r_kk)^2}}{\log X}\right),
$$
and therefore, using \eqref{eq:nearby} (recalling once again that $\mf{q}_k = 1$) and Lemma \ref{lem:UppBdMX}, that
$$
\mb{D}(f_u,n^{i\lambda u};X) \leq \mb{D}(f_u,n^{i\tilde{t}_u};X)  + \mb{D}(n^{i\lambda u}, n^{i\tilde{t}_u};X) \ll r_k k,
$$
and the claim follows.
\end{proof}
\begin{proof}[Proof of Proposition \ref{prop:gapConc}]
By Lemma \ref{lem:complambda},
$$
\int_{-1}^1 \mb{D}(f_u,n^{i\lambda u};X)^2 du \ll (r_k k)^2.
$$
Set $\lambda_0 := \lambda/2\pi$. We may explicitly compute the LHS as
$$
\int_{-1}^1 \mb{D}(f_u, n^{i\lambda u}; X)^2 du = \sum_{p \leq X} \frac{1}{p} \int_{-1}^1 (1- \cos(2\pi u(g(p)-\lambda_0 \log p))) du = 2\sum_{p \leq X} \frac{1}{p}\left(1-\text{sinc}(2\pi (g(p)-\lambda_0 \log p))\right),
$$
where we have written $\text{sinc}(t) := (\sin t)/t$ for $t \neq 0$, and $\text{sinc}(0) = 1$. From the estimate
$$
1-\text{sinc}(\alpha) \asymp \min\{1,\alpha^2\},
$$
we deduce that
$$
\int_{-1}^1 \mb{D}(f_u, n^{i\lambda u}; X)^2 du \asymp \sum_{p \leq X} \frac{\min\{1,(g(p)-\lambda_0 \log p)^2\}}{p} = U_g(X;\lambda_0).
$$
We deduce therefore that for every $k_0 \leq k \leq K$,
$$
U_g(X) = \min_{|\lambda'| \leq X^2} U_g(X;\lambda') \ll (r_k k)^2.
$$
Combined with \eqref{eq:rkEst}, this yields
\begin{equation}\label{eq:lambdaYkbd}
\lambda(Y_k) \ll \frac{k}{\sqrt{U_g(X)}}. 
\end{equation}
We now select $K := \sqrt{C\log\log\log X}$ and $R := \eta \sqrt{\log\log X}/K$ for $\eta,C > 0$ chosen suitably so that
$$
cC \geq 1, \quad \eta  < \frac{1}{400}.
$$
The latter of these conditions guarantees that $(r_k k)^2 < \tfrac{1}{400^2} \log\log X$ whenever $k_0 \leq k \leq K$. Combining \eqref{eq:lambdaYkbd} with \eqref{eq:tildeQgBd}, we deduce that
$$
\mc{Q}_g(X;d,h) \ll \frac{1}{\sqrt{U_g(X)}}\sum_{k_0 \leq k \leq K} e^{-ck^2}  k + e^{-cK^2}K + \frac{1}{R} \ll \frac{1}{\sqrt{U_g(X)}} + \left(\frac{\log\log \log X}{\log\log X}\right)^{1/2},
$$
as claimed. 
\end{proof}

\section{The finite concentration case: the proof of Theorem \ref{thm:PowBad}} \label{sec:finConc}
In the sequel, we continue to assume that $f$ is a multiplicative function taking values in $(0,\infty)$. As before, we fix $\omega \geq 1$ and assume that $\delta(\mc{L}_f(\omega)) \delta(\mc{S}_f(\omega)) = 0$. 
By Theorem \ref{thm:redtoFinConc}, the additive function $g = \log f$ has finite concentration. Therefore, there exists $\lambda \in \mb{R}$ such that
$$
\sum_p \frac{\min\{1,g_{\lambda}(p)^2\}}{p} < \infty.
$$
We fix $\lambda$ with the above property in the sequel. \\
To complete the proof of Theorem \ref{thm:PowBad} we will require a couple of ingredients, tailored to additive functions with finite concentration. In the sequel, given an additive function and $X \geq 1$ we shall write
$$
A_h(X) := \sum_{p^k \leq X} \frac{h(p^k)}{p^k}\left(1-\frac{1}{p}\right).
$$
\begin{lem} \label{lem:smallDev}
Let $h$ be a real-valued additive function, such that there are parameters $G,T \geq 1$ with:
\begin{itemize}
\item $|h(p)|\leq G$ for all $p$, 
\item $h(p^k) = 0$ whenever $k \geq 2$ and $p^k > T$, and
\item $\sum_{p} h(p)^2/p < \infty$.
\end{itemize}
Finally, assume that $x \geq x_0(G,T)$. Then, uniformly over $Q \leq \sqrt{x}$ and $1 \leq a \leq Q$ with $(Q,a) \leq T$,
$$
\frac{1}{x} \sum_{n \leq x} \left|h(Qn+a) - A_h(Qx) + \sum_{\ss{p | Q/(Q,a)}} \frac{h(p)}{p} \right|^2 \ll_{G,T} 1.
$$
\end{lem}
\begin{proof}
Define $\tilde{h}(n) := \sum_{p^k || n} h(p^k) 1_{p \leq \sqrt{x}}$, and set $Q' := Q/(Q,a)$, $a' := a/(Q,a)$. If $x > T^2$ then as $Qn+a \leq 2x^{3/2}$, we have 
$$
h(Qn+a) = \tilde{h}(Qn+a) + O_G(1) = \tilde{h}(Q'n+a') + O_{G,T}(1).
$$
Moreover, note that
$$
\left|A_h(Qx) - A_{\tilde{h}}(\sqrt{x})\right| \leq G\sum_{\ss{\sqrt{x} < p \leq x^{3/2}}} \frac{1}{p} \ll_G 1,
$$
as well as
$$
\sum_{\ss{p^k \leq x \\ k \geq 2, \ p|Q'}} \frac{h(p^k)}{p^k}\left(1-\frac{1}{p}\right) + \sum_{p|Q'} \frac{h(p)}{p^2} \leq \sum_{\ss{p^k \leq T \\ k \geq 2}} \frac{h(p^k)}{p^k}\left(1-\frac{1}{p}\right) + O_G(1) \ll_{G,T} 1.
$$
Therefore, writing
$$
A^{(Q')}_{\tilde{h}}(\sqrt{x}) := A_{\tilde{h}}(\sqrt{x}) - \sum_{\ss{p^k \leq x \\ p | Q'}} \frac{h(p^k)}{p^k}\left(1-\frac{1}{p}\right) = \sum_{\ss{p^k \leq x \\ p \nmid Q'}} \frac{h(p^k)}{p^k}\left(1-\frac{1}{p}\right)
$$
and applying Cauchy-Schwarz, it suffices to show that
\begin{align*}
\frac{1}{x}\sum_{n \leq x} \left|\tilde{h}(Q'n+a') - A^{(Q')}_{\tilde{h}}(\sqrt{x})\right|^2 
\ll_{G,T} 1.
\end{align*}
We now expand the square, noting again that if $x > T^2$ then $\tilde{h}(p^k) = 0$ for $p^k > \sqrt{x}$ and all $k \geq 1$. Thus,
\begin{align*}
&\sum_{\ss{p_1^{k_1}, p_2^{k_2} \leq \sqrt{x} \\ (p_1p_2,Q') = 1}} h(p_1^{k_1})h(p_2^{k_2}) \frac{1}{x}\sum_{\ss{n \leq x \\ p_j^{k_j}||Q'n+a', \ j = 1,2}} 1 - 2A^{(Q')}_{\tilde{h}}(\sqrt{x})\sum_{\ss{p^k \leq \sqrt{x} \\ p \nmid Q}} h(p^k) \cdot \left(\frac{1}{p^k}\left(1-\frac{1}{p}\right) + O(1)\right) + A^{(Q')}_{\tilde{h}}(\sqrt{x})^2\\
&= \sum_{\ss{p_1^{k_1},p_2^{k_2} \leq \sqrt{x} \\ (p_1p_2,Q') = 1}} h(p_1^{k_1})h(p_2^{k_2})\left(\frac{\prod_{p|p_1p_2}(1-1/p)}{[p_1^{k_1},p_2^{k_2}]} - \frac{\prod_{j = 1,2}(1-1/p_j)}{p_1^{k_1}p_2^{k_2}}\right) \\
&+ O\left(\frac{1}{x}\left(\sum_{p^k \leq \sqrt{x}} |h(p^k)|\right)\left(A_{|\tilde{h}|}(\sqrt{x}) + \sum_{p^k \leq \sqrt{x}} |h(p^k)|\right)\right).
\end{align*}
Using the facts that $|h(p)| \leq G$ and $h(p^k) = 0$ for $k \geq 2$ and $p^k > T$, the error term is $O_{G,T}\left(\tfrac{1}{(\log x)^2}\right)$. The main term cancels all but the terms for which $p_1 = p_2$, in which case we get an upper bound
$$
\sum_{p \leq \sqrt{x}} \frac{h(p)^2}{p} + \sum_{\ss{p^\ell\leq p^k \leq T \\ k \geq 2}} \frac{|h(p^k)h(p^\ell)|}{p^k} \ll_{G,T} 1.
$$
The claim follows.
\end{proof}
In the sequel, given parameters $T,G \geq 1$ we define
$$
W = W_{T,G} := \prod_{\ss{p \leq T \\ |g_{\lambda}(p)| > G}} p.
$$
\begin{lem}\label{lem:unifClose}
Assume that $g$ has finite concentration, and let $\lambda \in \mb{R}$ be its corresponding minimiser. Let $\e > 0$, fix $D \geq 1$ and let $G := 1+\max_{1 \leq d \leq D}|g_{\lambda}(d)| \geq 1$. Then there exist $T = T(\e,D) \geq D$ and $B = B(\e,D)$ such that if $x \geq x_0(\e,D)$ and $W = W_{T,G}$, 
\begin{align*}
\max_{\ss{Q_1,Q_2 \leq (\log x)^{1/2} \\ W|Q_j, \ j = 1,2}} \max_{\ss{1 \leq a_j \leq Q_j \\ (a_j,Q_j) \leq D, \ j = 1,2}} &\frac{1}{x}\left|\left\{n \leq x : \ \left|g_{\lambda}(Q_2n+a_2) + \sum_{\ss{p| \tfrac{Q_2}{(Q_2,a_2)} \\ |g_{\lambda}(p)| \leq G}} \frac{g_{\lambda}(p)}{p} - g_{\lambda}(Q_1n + a_1) - \sum_{\ss{p| \tfrac{Q_1}{(Q_1,a_1)} \\ |g_{\lambda}(p)| \leq G}} \frac{g_{\lambda}(p)}{p}\right| \leq B\right\}\right| \\
&\geq 1-\e.
\end{align*}
\end{lem}
\begin{proof}
We aim to apply Lemma \ref{lem:smallDev} with $Q = Q_j$ and $a = a_j$, for both $j = 1,2$.
In the sequel, we define $\tilde{g}_{\lambda}$ to be the additive function satisfying 
\begin{align*}
\tilde{g}_{\lambda}(p^k) := \begin{cases} g_{\lambda}(p) &\text{ if } |g_{\lambda}(p)| \leq G, \ k = 1, \\ g_{\lambda}(p^k) &\text{ if } p^k \leq T, \ k \geq 2, \\ 0 &\text{ otherwise.} \end{cases} 
\end{align*}
Since
$$
\sum_{p\leq y} \frac{\tilde{g}_{\lambda}(p)^2}{p} \leq G^2\sum_{p \leq y} \frac{\min\{g_{\lambda}(p)^2,1\}}{p} \ll_G 1, \quad y \geq 1,
$$
it follows that $\tilde{g}_{\lambda}$ satisfies all of the properties required to apply the previous lemma. To this end, we first show that
\begin{equation}\label{eq:gtoTildeg}
\max_{j=1,2}|\{n \leq x : \tilde{g}_{\lambda}(Q_jn + a_j)\neq g_{\lambda}(Q_jn+a_j)\}| \leq \frac{1}{2}\e x.
\end{equation}
To see this, note first of all that if $W|Q$ and $(Q,a) \leq D$, then the conditions $p|Qn+a$ and $|g_{\lambda}(p)| > G$ imply that $T < p \leq 2Qx$: indeed, $p \nmid D$ because $|g_{\lambda}(p)| > G$, and since $W | Q/(Q,a)$, $p\nmid W$ also. Now, if $T$ and $x$ are chosen suitably large as a function of $\e$ and $D$ then because of $Qx \leq x(\log x)^{1/2}$,
\begin{align*}
|\{n \leq x : \exists \ p > T, \ |g_{\lambda}(p)| > G, \ p|Qn+a\}| &\leq \sum_{\ss{T < p \leq 2Qx \\ |g_{\lambda}(p)| > G}} |\{n \leq x : p|Qn+a\}| \leq x \sum_{\ss{p > T \\ |g_{\lambda}(p)| > G}} \frac{1}{p} + \pi(Qx) \\
&\leq \left(x G^2 \sum_{p > T} \frac{\min\{g_{\lambda}(p)^2,1\}}{p} + O\left(\frac{Qx}{\log x}\right)\right) \\
&\leq x\left(\frac{\e }{16} + O\left(\frac{1}{\sqrt{\log x}}\right)\right) \leq \frac{\e x}{8}.
\end{align*}
Next, we observe that by making $T$ larger as needed, we similarly have
$$
|\{n \leq x : \exists p^k > T, \ k \geq 2, \ p^k|Qn+a\}| \leq x\sum_{\ss{p^k > T \\ k \geq 2}} \frac{1}{p^k} + O(\pi(\sqrt{Qx})) \leq \frac{\e x}{8}.
$$
By the union bound, it follows that the number of $n \leq x$ such that for one of $j = 1,2$ we have either $p|(Q_jn+a_j)$ with $p > T$ and $|g_{\lambda}(p)| > G$, or $p^k ||(Q_jn+a_j)$ with $p^k > T$ and $k \geq 2$, is $\leq \e x/4$.
Now, if $n$ does not belong to either of these sets, for \emph{both} of the pairs $Q_j,a_j$ then we have that $p^k||(Qn+a)$ implies that either $k = 1$, $|g_{\lambda}(p)| \leq G$, or else $k \geq 2$ and $p^k \leq T$. It follows, therefore, that for all such $n$ we have $g_{\lambda}(Qn+a) = \tilde{g}_{\lambda}(Qn+a)$, again for both pairs of $Q$ and $a$, and \eqref{eq:gtoTildeg} follows by the union bound. \\
Now, let $B$ also be chosen large in terms of $\e$ and $D$. By the pigeonhole principle,
\begin{align} \label{eq:largeDevQj}
&\frac{1}{x}|\{n \leq x : \ \left|g_{\lambda}(Q_2n+a_2) + \sum_{\ss{p| \tfrac{Q_2}{(Q_2,a_2)} \\ |g_{\lambda}(p)| \leq G}} \frac{g_{\lambda}(p)}{p} - g_{\lambda}(Q_1n + a_1) - \sum_{\ss{p| \tfrac{Q_1}{(Q_1,a_1)} \\ |g_{\lambda}(p)| \leq G}} \frac{g_{\lambda}(p)}{p}\right| \geq B\}| \nonumber\\
&\leq \frac{1}{x}|\{ n \leq x : \ \left|\tilde{g}_{\lambda}(Q_2n+a_2) + \sum_{\ss{p| \tfrac{Q_2}{(Q_2,a_2)}}} \frac{\tilde{g}_{\lambda}(p)}{p} - \tilde{g}_{\lambda}(Q_1n + a_1) - \sum_{\ss{p| \tfrac{Q_1}{(Q_1,a_1)}}} \frac{\tilde{g}_{\lambda}(p)}{p}\right| \geq B\}| + \frac{\e}{2} \nonumber\\
&\leq \frac{2}{x} \max_{j = 1,2}|\{n \leq x : \left|\tilde{g}_{\lambda}(Q_jn+a_j) - \mc{A}_{\tilde{g}_{\lambda}}(Q_jx) + \sum_{\ss{p|\tfrac{Q_j}{(Q_j,a_j)}}} \frac{\tilde{g}_{\lambda}(p)}{p}\right| \geq B/3\}| + \frac{\e}{2},
\end{align}
using the fact that
$$
\left|A_{\tilde{g}_{\lambda}}(Q_2 x) - A_{\tilde{g}_{\lambda}}(Q_1 x)\right| \leq G\sum_{Q_1 x < p \leq Q_2x} \frac{1}{p} \leq 2G\frac{\log\log x}{\log x} < \frac{B}{6},
$$
say, whenever $Q_1 \leq Q_2 \leq \sqrt{\log x}$, and $x \geq x_0(G,T)$ sufficiently large so that the above and $Q_1 x > T$ both hold. 
Thus, taking $B$ sufficiently large in terms of $\e$ and $D$ then (since $T$ depends only on these parameters) by Chebyshev's inequality and Lemma \ref{lem:smallDev} we get that the RHS of \eqref{eq:largeDevQj} is
$$
\leq \max_{j = 1,2}\frac{9}{B^2 x} \sum_{n \leq x} \left|\tilde{g}_{\lambda}(Q_jn+a_j) - \mc{A}_{\tilde{g}_{\lambda}}(Q_jx) + \sum_{\ss{p|\tfrac{Q_j}{(Q_j,a_j)}}} \frac{\tilde{g}_{\lambda}(p)}{p}\right|^2 + \frac{\e}{2} \leq \e.
$$
Since all of these estimates held uniformly over $Q_j \leq (\log x)^{1/2}$ with $W|Q_j$, the claim follows.
\end{proof}
Finally, we need the following simple observation. By abuse of notation, given a real-valued additive function $h$ we write $\mc{B}_h$ to denote the set of powerfully bad primes for the positive multiplicative function $e^h$.
\begin{lem} \label{lem:easyVer}
Let $h$ be a real-valued additive function. Assume that there exists $C > 0$ such that for any pair of coprime $q_1,q_2 \in \mb{N}$, 
\begin{equation}\label{eq:minAssump}
\min\{|h(q_1)|,|h(q_2)|\} \leq C.
\end{equation}
Then $|\mc{B}_h| \leq 1$. Moreover, there exists $P \in \mb{N}$ with $\Omega(P) \leq 1$ such that
$$
\sup_{(n,P) = 1} |h(n)| < \infty.
$$
\end{lem}
\begin{proof}
If $p_1,p_2 \in \mc{B}_h$ are distinct primes then 
$$
\sup_{k \geq 1} |h(p_1^k)| = \infty = \sup_{\ell \geq 1} |h(p_2^\ell)|.
$$
There must thus exist $k_0,\ell_0$ such that the coprime integers $q_1 := p_1^{k_0}$ and $q_2 := p_2^{\ell_0}$ satisfy
$$
\min\{|h(q_1)|,|h(q_2)|\} > C,
$$
a contradiction. It follows that $|\mc{B}_h| \leq 1$. \\
Suppose next that 
\begin{equation}\label{eq:allPrimePowBdd}
\sup_{p \in \mb{P}} \sup_{k \geq 1} |h(p^k)| \leq C.
\end{equation}
If $\sup_n |h(n)| \leq C$ then we are done, taking $P = 1$. Otherwise, there exists $n_0$ such that $|h(n_0)| > C$. 
Now, take any $n \in \mb{N}$, and write $n = n'd$, where $p | (n,n_0) \Rightarrow p|d$ and thus $(n',n_0) = 1$. By \eqref{eq:minAssump}, $|h(n')| \leq C$. Thus, by additivity and \eqref{eq:allPrimePowBdd},
$$
|h(n)| \leq |h(n')| + |h(d)| \leq C + \omega(n_0) \sup_{p \in \mb{P}} \sup_{k \geq 1} |h(p^k)| \leq (1+\omega(n_0))C.
$$
Since this bound holds uniformly in $n$, the second claim of the lemma follows with $P = 1$.\\
It remains to show the final claim when \eqref{eq:allPrimePowBdd} does not hold. It is clear from \eqref{eq:minAssump} that there must be precisely one prime $p_0$ such that
$$
L := \sup_{k \geq 1} |h(p_0^k)| > C.
$$
We may therefore select $k$ such that
$|h(p_0^k)| > C$.
Using \eqref{eq:minAssump} once again, we see that for any $n$ coprime to $p_0$,
$$
|h(p_0^k)| > C \geq \min\{|h(n)|,|h(p_0^k)|\},
$$
so that $|h(n)| \leq C$. Thus, we find that
$$
\sup_{(n,p_0) = 1} |h(n)| \leq C.
$$
Now, if $L = \infty$ then $\mc{B}_h = \{p_0\}$ and $P := p_0$ in the second conclusion. Otherwise, suppose $L < \infty$. Take any $n \in \mb{N}$, and factor it as $n = m p_0^\nu$ with $p_0 \nmid m$. We obtain immediately that
$$
|h(n)| \leq |h(m)| + |h(p_0^\nu)| \leq C + L < \infty
$$
and we may take $P = 1$ in the second claim.
\end{proof}
\begin{proof}[Proof of Theorem \ref{thm:PowBad}] 
Recall that $f: \mb{N} \ra (0,\infty)$ is a multiplicative function such that $g = \log f$ has finite concentration. Let $\omega \geq 1$ (the result in the full range $\omega > 0$ follows from this restricted range by Remark \ref{rem:MonoOmega}), and assume that $\delta(\mc{L}_f(\omega)) \delta(\mc{S}_f(\omega)) = 0$. By \eqref{eq:ftoreciprocal}, replacing $f$ by $1/f$ if necessary, we may assume that $\delta(\mc{L}_f(\omega))=0$. 

We consider two cases. \\
\underline{Case 1:} We assume first of all that $g_{\lambda}$ takes unbounded \emph{positive} values, i.e.,
$$
\sup_{q \geq 1} g_{\lambda}(q) = \infty.
$$
Fix $D = 1$, so $G = 1$, and $\e = 1/3$. Let also $T = T(\e,D) \geq 1$ and $B = B(\e,D) \geq 1$ be those parameters appearing in Lemma \ref{lem:unifClose}. It follows that if $x \geq x_0$ for $x_0$ depending at most on $g$ and $W = W_{T,1}$, then we may select $q$ such that
$$
qW \leq (\log x)^{1/4}, \quad g_{\lambda}(q) \geq 1+B + \max_{d|W}|g_{\lambda}(d)| +\Delta.
$$ 
Set $\tilde{W} := W/(q,W)$ and take $Q_1 = q\tilde{W}$, $a_1 = 1$ and $Q_2 = Q_1^2$, $a_2 = q\tilde{W} \pm 1$, noting that $(a_j,Q_j) = 1$ for $j = 1,2$, and $Q_1 \triangle Q_2 = 1$. By Lemma \ref{lem:unifClose}, we see that for each $\eta \in \{-1,+1\}$, for a proportion $\geq 2/3$ of integers $n \leq x$ we have
$$
|g_{\lambda}(\tilde{W}qn + 1) - g_{\lambda}(\tilde{W}q(\tilde{W}q n+1) + \eta)| \leq B.
$$
By inclusion-exclusion, it follows that there is a proportion $\geq 1/6$ of $x/6 < n \leq x$ for which
$$
\max_{\eta \in \{-1,+1\}}|g_{\lambda}(\tilde{W}qn + 1) - g_{\lambda}(\tilde{W}q(\tilde{W}q n+1) + \eta)| \leq B
$$
But for each of these $n$, if we write $m = \tilde{W}qn+1$ it follows that $(\tilde{W},q,m) = 1$, and by additivity,
\begin{align*}
g_{\lambda}(\tilde{W}qm) &\geq g_{\lambda}(q) - |g_{\lambda}(\tilde{W})| + g_{\lambda}(m) \geq 1+B + \Delta + \max\{g_{\lambda}(Wqm - 1), g_{\lambda}(\tilde{W}qm+1)\} - B \\
&= 1 + \Delta + \max\{g_{\lambda}(\tilde{W}qm - 1), g_{\lambda}(\tilde{W}qm+1)\}.
\end{align*}
Adding $\lambda \log(\tilde{W}qm)$ to both sides and noting that $m \gg x$, it follows that if $x$ is large enough we get
$$
g(\tilde{W}qm) \geq \frac{1}{2} + \Delta + \max\{g(\tilde{W}qm-1),g(\tilde{W}qm+1)\} > \Delta + \max\{g(\tilde{W}qm-1),g(\tilde{W}qm+1)\}.
$$
Since this holds for a positive proportion of $n \leq x$, we obtain a contradiction to $\delta(\mc{L}_f(\omega)) = 0$. \\
\underline{Case 2:} We may assume now that
\begin{equation}\label{eq:bddPos}
\sup_{q \geq 1} g_{\lambda}(q) < \infty.
\end{equation}
We claim that there exists a threshold $C> 0$ depending only on $\omega$ and $g$, such that for any two coprime $q_1,q_2 \in \mb{N}$,
\begin{equation}\label{eq:reduced}
\min\{-g_{\lambda}(q_1),-g_{\lambda}(q_2)\} \leq C.
\end{equation}
Assume this for the moment. Given our assumption \eqref{eq:bddPos}, it follows, necessarily, that for some possibly larger threshold $C' > 0$, for every pair of coprime $q_1,q_2 \in \mb{N}$ we in fact have
$$
\min\{|g_{\lambda}(q_1)|,|g_{\lambda}(q_2)|\} \leq C'.
$$
By Lemma \ref{lem:easyVer}, we obtain the desired claims: 
\begin{itemize}
\item $|\mc{B}_{g_{\lambda}}| \leq 1$, and 
\item there is $P \in \mb{N}$ with $\Omega(P) \leq 1$ such that $\sup_{(n,P) = 1} |g_{\lambda}(n)| < \infty$. \end{itemize}
It therefore remains to show \eqref{eq:reduced}, which we will do by contradiction. \\
In preparation, set $D = 24$, $G := 1 + \max_{1 \leq d \leq D} |g_{\lambda}(d)|$, $\e= 1/3$, $B = B(\e,D)$ and $T = T(\e,D) \geq 24$ chosen as in Lemma \ref{lem:unifClose}, and set $W = W_{T,G}$. Take $C := 1 + G + B + \Delta$, and assume that \eqref{eq:reduced} does not always hold, i.e., we can find coprime integers $q_1,q_2 \in \mb{N}$ such that
$$
\min\{-g_{\lambda}(q_1), -g_{\lambda}(q_2)\} > C.
$$
 We choose $n = \tilde{W}(q_1q_2)^2m + b$, where $\tilde{W} := W/(W,q_1q_2)$ and $1 \leq b \leq \tilde{W}(q_1q_2)^2$ is chosen so that 
$$
b \equiv 1 + q_1 \pmod{q_1^2}, \quad b \equiv -1 + q_2 \pmod{q_2^2}, \quad (b^3-b,\tilde{W}/(6,\tilde{W})) = 1.
$$
Note that the latter condition can be guaranteed by the Chinese remainder theorem, since it holds modulo each prime $p|\tilde{W}$, $p > 3$. The first two conditions are then compatible with this choice because the squarefreeness of $W$ implies that $(\tilde{W},(q_1q_2)^2) = 1$. \\
We now observe that
\begin{equation}\label{eq:GCDcond}
\max\{(b,\tilde{W}(q_1q_2)^2) , ((b-1)/q_1,\tilde{W}q_1q_2^2), ((b+1)/q_2,\tilde{W}q_2q_1^2)\} \leq D.
\end{equation}
Indeed, since $(b,q_1q_2) = ((b-1)/q_1,q_1) = ((b+1)/q_2,q_2) = 1$ and $(b^3-b,\tilde{W}) \leq 6$, it suffices to show that
\begin{equation}\label{eq:redGCD}
\left(\frac{b-1}{q_1}, q_2^2\right)\left(\frac{b+1}{q_2},q_1^2\right) = (b-1,q_2) (b+1,q_1) \leq 4.
\end{equation}
Since $q_2|(b+1)$, $(b-1,q_2)$ divides $(b-1,b+1)$, which divides $2$. The same argument shows that $(b+1,q_1) \leq 2$. Since $q_1,q_2$ are coprime, at most one of these is equal to $2$, so as $(d_1,d_2^2)$ divides $(d_1,d_2)^2$ for all $d_1,d_2 \in \mb{N}$, \eqref{eq:redGCD} follows. \\
We now take $Q_1 := \tilde{W}(q_1q_2)^2$, $a_1 = b$, $Q_2 := \tilde{W}q_1q_2^2$, $a_2 = (b-1)/q_1$ and $Q_3 := \tilde{W}q_1^2q_2$, $a_3 = (b+1)/q_2$, so that $(Q_j,a_j) \leq D$ for all $j =1,2,3$. We may assume that $x$ is large enough that $(q_1q_2)^2 \tilde{W} \leq \sqrt{\log x}$. Arguing as in the previous proof, with the aforementioned choices of $T$ and $B$ there is a proportion $\geq 1/6$ of integers $x/6 < m \leq x$ such that 
\begin{align*}
\max\{|g_{\lambda}(Q_1m + a_1) - g_{\lambda}(Q_2m + a_2)|, |g_{\lambda}(Q_1m + a_1) - g_{\lambda}(Q_3m + a_3)|\} \leq B + M,
\end{align*}
where the additional $M$ comes from trivially bounding the terms 
$$
\sum_{\ss{p|\tfrac{Q_1}{(Q_1,a_1)} \\ |g_{\lambda}(p)|\leq M}} \frac{g_{\lambda}(p)}{p} - \sum_{\ss{p|\frac{Q_j}{(Q_j,a_j)} \\ |g_{\lambda}(p)|\leq M}} \frac{g_{\lambda}(p)}{p} \leq M\sum_{p|24} \frac{1}{p} < M, \quad j=2,3.
$$
Since $q_2(Q_3m + a_3) = Q_1m+a_1 + 1$ and $q_1(Q_2m + a_2) = Q_1m+a_1 - 1$, we find that
\begin{align*}
&\max\{g_{\lambda}(Q_1m + a_1-1), g_{\lambda}(Q_1m + a_1 + 1)\} + \Delta \\
&\leq \max\{g_{\lambda}(q_1) +  g_{\lambda}(Q_2m + a_2), g_{\lambda}(q_2) + g_{\lambda}(Q_3m + a_3)\} + \Delta \\
&\leq -(1+M+B +\Delta) + g_{\lambda}(Q_1m + a_1) + B + M + \Delta = g_{\lambda}(Q_1m+a_1) - 1. 
\end{align*}
Adding $\lambda \log(Q_1m+a_1)$ to both sides once again, and using the fact that $m > x/6$, we deduce that when $x$ is large enough,
$$
\max\{g(Q_1m + a_1-1), g(Q_1m + a_1 + 1)\} + \Delta \leq g(Q_1m+a_1) - 1/2 < g(Q_1m+a_1).
$$
Since this holds for a positive proportion of $m \leq x$, as $x \ra \infty$ this contradicts the hypothesis that $\delta(\mc{L}_f(\omega)) = 0$. Thus, \eqref{eq:reduced} holds after all, and as discussed earlier the proof is complete.
\end{proof}

\section{Proof of Theorems \ref{thm:SarkCM} and \ref{thm:multAtMost1}} \label{sec:apps}
Theorem \ref{thm:SarkCM} now follows easily from Theorem \ref{thm:PowBad}.
\begin{proof}[Proof of Theorem \ref{thm:SarkCM}]
Assume that $f$ is completely multiplicative and takes positive integer values. As remarked earlier, it suffices to take $\omega \geq 1$. Assume that $\delta(\mc{L}_f(\omega))\delta(\mc{S}_f(\omega)) = 0$. Theorem \ref{thm:redtoFinConc} allows us to assume that $g = \log f$ satisfies
$$
\sum_p \frac{\min\{1,g_{\lambda}(p)^2\}}{p} < \infty,
$$
for some $\lambda \in \mb{R}$. \\
By Theorem \ref{thm:PowBad}, we see that there exists an integer $P$ with at most a single prime factor, such that 
$$
\sup_{\ss{n \in \mb{N} \\ (n,P) = 1}} |g_{\lambda}(n)| < \infty.
$$
But as $g$ is completely additive, if $p \nmid P$ then $g_{\lambda}(p) = 0$, since otherwise
$$
\limsup_{k \ra \infty} |g_{\lambda}(p^k)| = |g_{\lambda}(p)| \limsup_{k \ra \infty} k = \infty,
$$
a contradiction. Therefore, $f(p) = p^{\lambda}$ for all but at most a single prime. This will yield a contradiction to Sark\H{o}zy's hypothesis, provided we can show that $\lambda \in \mb{N}_0$. \\
To begin with, since $f(p) \geq 1$ we clearly have $\lambda \geq 0$. 
Next, suppose for contradiction that $\lambda \notin \mb{Q}$. By uniqueness of prime factorisation,
$\{\log p_1,\log p_2,\log p_3\}$ is linearly independent over $\mb{Q}$. Since $\{1,\lambda\}$ is also linearly independent over $\mb{Q}$, the six exponentials theorem \cite[Sec. 2.1]{Lang}
implies that one of 
$$
\{p_1,p_2,p_3, p_1^{\lambda},p_2^{\lambda}, p_3^{\lambda}\}
$$
is transcendental.
This contradicts the fact that $f(p) \in \mb{N}$. Thus, we find that $\lambda = a/b$ for some coprime integers $a,b$ with $b \geq 1$ and $a \geq 0$. \\
If $a = 0$ then $\lambda = 0 \in \mb{N}_0$ as required. Otherwise, suppose $a \geq 1$. If $b > 1$ then from the identity $f(p)^b = p^a$ and unique factorisation of integers, we deduce that $f(p) = p^\nu$ for some $\nu \geq 1$. It follows that $a = \nu b$, so that $a$ and $b$ are not coprime, a contradiction. This implies that $b = 1$ as claimed, and so $\lambda \in \mb{N}$. The conclusion follows.
\end{proof}
It remains to show Theorem \ref{thm:multAtMost1}, in the case $0 < \omega \leq 1$.
Our key result is the following.
\begin{prop}\label{prop:omega1}
Let $f$ be a positive-valued multiplicative function such that $g = \log f$ has finite concentration with minimiser $\lambda$. Assume that $|\mc{B}_{g_{\lambda}}| \leq 1$, and that there is $P \in \mb{N}$ with $\Omega(P) \leq 1$ such that $\sup_{(n,P) = 1}|g_{\lambda}(n)| < \infty$. Then either:
\begin{enumerate}[(i)]
\item $\mc{L}_f$ and $\mc{S}_f$ both have positive upper logarithmic density, or
\item there is $c \in \mb{R}$ such that $f(p) = p^c$ for all but at most one $p$.
\end{enumerate}
\end{prop}
\begin{proof}[Proof of Theorem \ref{thm:multAtMost1} assuming Proposition \ref{prop:omega1}]
Assume that $f: \mb{N} \ra \mb{N}$ is multiplicative, and such that for each $k \in \mb{N}_0$ there are at least two primes $p$ with $f(p) \neq p^k$. Assume that $\delta(\mc{L}_f)\delta(\mc{S}_f) = 0$. As above, Theorem \ref{thm:redtoFinConc} allows us to assume that $g = \log f$ has finite concentration with parameter $\lambda$, and Theorem \ref{thm:PowBad} allows us to assume that $|\mc{B}_{g_{\lambda}}| \leq 1$, and the existence of $P \in \mb{N}$ with $\Omega(P) \leq 1$ such that $\sup_{(n,P) = 1} |g_{\lambda}(n)| < \infty$. By Proposition \ref{prop:omega1}, there is $c \in \mb{R}$ such that we have $f(p) = p^c$ for all but one prime $p$. Since $f(p) \in \mb{N}$ for all primes $p$, the argument from the proof of Theorem \ref{thm:SarkCM} shows that $c \in \mb{N}_0$, and we obtain a contradiction. The claim follows. 
\end{proof}
It therefore remains to prove Proposition \ref{prop:omega1}. To prove this we will use the following two facts. The first is due to K\'{a}tai \cite{Katai78}.
\begin{thm}[K\'{a}tai]\label{thm:KataiLog}
Let $h$ be a real-valued additive function. If
$$
\liminf_{x \ra \infty} \frac{1}{\log x} \sum_{n \leq x} \frac{1}{n}|h(n+1)-h(n)| = 0
$$
then there exists $c \in \mb{R}$ such that $h(n) = c\log n$ for all $n$.
\end{thm}
The second fact is a special case of a result due to the author \cite{MangerelEq}.
\begin{thm} \label{thm:consec}
Let $F$ be a positive-valued multiplicative function. Assume that
\begin{equation}\label{eq:densefnot1}
\sum_{p : F(p) \neq 1} \frac{1}{p} = \infty.
\end{equation}
Then the set $\{n \in \mb{N} : F(n) = F(n+1)\}$ has logarithmic density $0$.
\end{thm}
\subsection{When \eqref{eq:densefnot1} does not hold}
To apply Theorem \ref{thm:consec}, we must first reduce to the case that \eqref{eq:densefnot1} holds.
Key to this reduction is the following simple device for producing large sets with prescribed divisibility conditions, relative to a ``sparse'' set of primes.  
\begin{lem}\label{lem:fix23q}
Fix $q \in \mb{N}$ squarefree with $(q,6) = 1$, and let\footnote{We allow $j_2 = 2$ in order to obtain $2||n$.} $j_1 \in \{-1,0,1\}, j_2 \in \{-1,0,1,2\}$. Let $S$ be a set of primes such that
$$
\sum_{p \in S} \frac{1}{p} < \infty.
$$
Then there is a positive density set $\mc{A} = \mc{A}_{q,j_1,j_2}$ such that for every $n \in \mc{A}$: 
\begin{enumerate}[(a)]
\item $p|n(n-1)(n+1)$ implies that $p \in \{2,3\}$, $p | q$ or $p \notin S$ 
\item if $p^2|n(n-1)(n+1)$ then $p= 2$,
\item $q|n$ and $(n/q,q) = 1$,
\item $3||(n+j_1)$ and $4||(n+j_2)$.
\end{enumerate}
Similarly, given primes $p_1<p_2$ and $j \in \{-1,1\}$ there is a positive density set $\mc{B} = \mc{B}_{p_1,p_2,j}$ such that for every $n \in \mc{B}$, the following holds: 
\begin{enumerate}
\item $p|n(n-1)(n+1)$ implies that $p \in \{2,3,p_1,p_2\}$, or $p\notin S$,
\item if $p^2|n(n-1)(n+1)$ then $p=2$, 
\item $p_1||(n-1)$ and $p_2||(n+1)$, and
\item $4||(n+1)$, $3||(n+j)$.
\end{enumerate}
\end{lem}
\begin{proof}
We begin with the construction of $\mc{A}$. Let $\e > 0$ and $M_0 = M_0(\e)>q$ be parameters to be chosen shortly. Define 
$$
\mf{m} := \prod_{\ss{5 \leq p \leq M_0 \\ p \nmid q}} p.
$$
We choose $n$ to satisfy the following system of congruences:
\begin{align*}
n \equiv -j_1 + 3 \pmod{9}, &\quad n \equiv -j_2 + 4 \pmod{8}, \\
n \equiv 2 \pmod{\mf{m}}, &\quad n \equiv q \pmod{q^2}.
\end{align*}
These are compatible since $\mf{m}$ is coprime to $6q$, and $q$ is coprime to $6$. For any solution $n$ to this system, it is clear that (c) and (d) hold. Moreover, since $(6,\mf{m}) = 1$, all three of $n,n-1$ and $n+1$ take invertible residue classes modulo $\mf{m}$, so that each of these is coprime to all of the primes $5 \leq p \leq M$ (with only $n$ being divisible by $q$). \\
Set $Q := 72q^2 \mf{m}$, and write any $n$ above in the form $n = Qk + c$, for some $c \pmod{Q}$ satisfying all of the required congruences. The number of $k \leq x$ for which there exists $p > M$ with $p \in S$ such that $p|(Qk+c+j)$ for some $j \in \{-1,0,1\}$ is
$$
\leq 3\max_{j \in \{-1,0,1\}} \sum_{\ss{M_0 < p \leq Q(x+1) \\ p \in S}} \sum_{\ss{k \leq x \\ k \equiv -Q^{-1}(c + j) \pmod{p}}} 1 \leq 3\sum_{\ss{M_0 < p \leq Q(x+1) \\ p \in S}} \left(\frac{x}{p} + 1\right) \leq 3x\sum_{\ss{p > M_0 \\ p \in S}} \frac{1}{p} + o(1),
$$
as $x \ra \infty$. In particular, we can choose $M_0$ large enough such that this is $\leq \e x/2$. Next, the number of $k \leq x$ such that there exists $p > M_0$ such that $p^2|(Qk+c+j)$ for some $j \in \{-1,0,1\}$ is similarly
$$
\leq 3\max_{j \in \{-1,0,1\}} \sum_{\ss{M_0 < p \leq \sqrt{Q(x+1)}}} \sum_{\ss{k \leq x \\ k \equiv -Q^{-1}(c + j) \pmod{p^2}}} 1 \leq 3\sum_{\ss{M_0 < p \leq \sqrt{Q(x+1)} }} \left(\frac{x}{p^2} + 1\right) \leq 3x\sum_{\ss{p > M_0}} \frac{1}{p^2} + o(1).
$$
Picking $M_0$ larger as needed, we can ensure that this is also $\leq \e x/2$. Thus, aside from $\leq \e x$ exceptional $k \leq x$, we can guarantee that if $n = Qk + c$ with $Q$ and $c$ as above, $n-1,n$ and $n+1$ have no square factors $p^2$ with $p > M_0$ and no prime factors from $S$ with $p > M_0$. Taking $\e = 1/2$, say, and $x \ra \infty$, the set of admissible $k$ must have density $\geq 1/2$.\\
But now we note that $n(n-1)(n+1)$ is coprime to all primes $5 \leq p \leq M_0$ besides $p \in \{2,3\}$ and, in the case of $n$, $p|q$. Therefore, $n-1,n+1$ have no square factors $p^2$ with $p > 3$, and since $(n/q,q) = 1$, $n$ similarly has no square factors $p^2$ with $p > 3$. This in fact improves to all $p > 2$ because $3||(n+j_1)$ so none of $n,n-1,n+1$ can be divisible by $9$. Part (b) thus follows. Similarly, $n-1,n+1$ are both indivisible by $p \in S$ with $p > 3$, and similarly for $n$ outside of the primes that divide $q$. Thus, (a) holds as well, and the claim follows with $\mc{A}$ the set of all non-exceptional $n$ thus constructed. \\
The construction of $\mc{B}$ is similar, but simpler. This time, let 
$$
\mf{m}' := \prod_{\ss{5 \leq p \leq M \\ p \neq p_1,p_2}} p. 
$$
As above, we choose $n$ according to a system of congruences, as follows: 
\begin{align*}
n \equiv 1 + p_1 \pmod{p_1^2}, \quad n \equiv -1 + p_2 \pmod{p_2^2}, \\
n \equiv 2 \pmod{\mf{m}'}, \quad n \equiv -j+3 \pmod{9}, \quad n \equiv 3 \pmod{8}.
\end{align*}
Then (3) and (4) hold immediately. This, time, setting $Q' := 72p_1^2p_2^2\mf{m}'$ and choosing $c'$ to be a residue class modulo $Q'$ that satisfies the above system, the argument from above shows that aside from $\leq \e x$ exceptional $k \leq x$, we can choose $n = Q'k + c'$ so that none of $n,n-1$ or $n+1$ has a square factor $p^2$ with $p > M$, and no prime factors $p \in S$ with $p > M$. As above, this implies that $n$ satisfies the conditions (1) and (2) as well. Taking $x \ra \infty$ and letting $\mc{B}$ be the set of all such non-exceptional $n$, we obtain that $\mc{B}$ has positive density, as required.
\end{proof}
We use Lemma \ref{lem:fix23q} to prove the following reduction to the case that $\{p : f(p) \neq 1\}$ is not sparse.
\begin{lem} \label{lem:fequal1}
Let $f: \mb{N} \ra (0,\infty)$ be a multiplicative function, such that 
$$
\sum_{p : f(p) \neq 1} \frac{1}{p} < \infty. 
$$
Suppose that there are at least two primes $p$ such that $f(p) \neq 1$. Then $\mc{L}_f$
has positive natural density.
\end{lem}
\begin{proof}
Define the sets
$$
S_+ := \{p : \ f(p) > 1\}, \quad S_- := \{p : \ f(p) < 1\}.
$$
By assumption, $S := S_+ \cup S_-$ has size $|S| \geq 2$. We consider two cases.\\ \\
\textbf{Case 1:} $|S_+| \geq 1$, or $f(4) > 1$. \\
Let us note the following general consequence of Lemma \ref{lem:fix23q}. If $n \in \mc{A}_{q,j_1,j_2}$ in the notation there, then $q|n$, and $f(p) = 1$ for all $(p,6) = 1$ with $p|(n/q)(n-1)(n+1)$. Furthermore, there are integers $a,b,c$ with $(a,q) = (b,3) = (c,2) = 1$, such that
$$
n = qa, \quad n+j_1 = 3b, \quad n+j_2 = 4c.
$$
Finally, the only square factor of any of $n,n-1$ or $n+1$ is $4$. It is therefore immediate that $f(n-1),f(n+1),f(n)$ depend only on $f(q)$ and the values $f(2),f(3)$ and $f(4)$. \\
Given this, we now consider the following subcases of Case 1, showing that in each one there is a choice of parameters $q \in \mb{N}$ coprime to $6$, $j_1 \in \{-1,0,1\}$ and $j_2 \in \{-1,0,1,2\}$ such that $\mc{A}_{q,j_1,j_2}$ is a subset of $\mc{L}_f$. This will thus imply that $\mc{L}_f$ has positive density, as claimed. \\
\underline{Case 1.1:} $2,3 \in S_+$. \\
Take $q = 1$, $j_1 = 0$, $j_2 = 2$, so that $2||n$ and $3||n$, while both $n-1,n+1$ are coprime to $6$. We thus find that
$$
f(n) = f(2)f(3) > 1, \ f(n-1) = f(n+1) = 1,
$$
as required. \\
\underline{Case 1.2:} $2 \in S_+$, $3 \notin S_+$. \\
Take $q = 1$, $j_1 = -1$, $j_2 = 2$, so that $2||n$ and $3|(n-1)$. Then $n-1,n+1$ are both coprime to $2$, and $n+1$ is coprime to $3$. We thus find that
$$
f(n) = f(2) > 1, \ f(n-1) = f(3) \leq 1, \ f(n+1) = 1.
$$
\underline{Case 1.3:} $f(4) > 1$, $3 \notin S_+$. \\
Take $q = 1$, $j_1 = -1$ and $j_2 = 0$. Then $4||n$, $3||(n-1)$, $n-1$ is odd and $n+1$ is coprime to $6$. We obtain
$$
f(n) = f(4) > 1, \ f(n-1) = f(3) \leq 1, \ f(n+1) = 1.
$$
\underline{Case 1.4:} $f(4) > 1$, $3 \in S_+$. \\
Take $q = 1$, $j_1 = j_2 = 0$. Then $4||n$, $3||n$ and $n-1,n+1$ are both coprime to $6$. We obtain
$$
f(n) = f(4)f(3) > 1, \ f(n-1) = f(n+1) = 1.
$$
\underline{Case 1.5:} $f(4),f(2)\leq 1$, $3 \in S_+$. \\
Take $q = 1$, $j_1 = 0$ and $j_2 = -1$. Then $4||(n-1)$, $2||(n+1)$, $3||n$, and both $n-1,n+1$ are coprime to $3$. We obtain
$$
f(n) = f(3) > 1, \ f(n-1) = f(4) \leq 1, \ f(n+1) = f(2) \leq 1.
$$
\underline{Case 1.6:} $f(2),f(3),f(4) \leq 1$, $p \in S_+$ for some $p > 3$. \\
Take $q = p$, $j_1 = -1$ and $j_2 = 1$. Then $3||(n-1)$, $4||(n+1)$, and $n$ is coprime to $6$. We obtain
$$
f(n) = f(p) > 1, \ f(n-1) = f(2)f(3) \leq 1, \ f(n+1) = f(4) \leq 1.
$$
The above subcases are exhaustive, so the proof of the lemma in Case 1 is complete. \\ \\
\textbf{Case 2:} $S_+ = \emptyset$ and $f(4) \leq 1$\\
Since $|S| = |S_-| \geq 2$ we may pick minimal, distinct primes $p_1 < p_2$ from $S_-$. We show that there is a suitable choice of $j \in \{-1,1\}$ such that, using Lemma \ref{lem:fix23q}, $\mc{B}_{p_1,p_2,j}$ is a positive density subset of $\mc{L}_f$. We consider two subcases. The general consequence of $n \in \mc{B}_{p_1,p_2,j}$ is that $n,n-1$ and $n+1$ has no square factor besides $4$, and that $4||(n+1)$. Also, $n$ has no prime factors from $S$ besides possibly for $2,3,p_1$ and $p_2$. In particular, the values of $f(n)$, $f(n-1)$, $f(n+1)$ are determined by $f(2),f(3),f(4),f(p_1)$ and $f(p_2)$. \\
\underline{Case 2.1:} $2,3 \in S_-$ \\
Here we take $j = 1$, $p_1 = 2$ and $p_2 = 3$, so that $2||(n-1)$ and $3||(n+1)$, which is compatible with the conditions $4||(n+1)$ and $3||(n+j)$. Since $n$ is coprime to $6$, the choice of $p_1,p_2$ guarantees that $n$ is squarefree and has no prime factors in $S$. Thus, $f(n) = 1$. On the other hand, we have $2||(n-1)$ and as $S_+$ is empty, $f((n-1)/2),f((n+1)/12) \leq 1$. It follows that
\begin{align*}
f(n-1) &= f(2) f((n-1)/2) \leq f(2) < 1 \\
f(n+1) &= f(3)f(4) f((n+1)/12) \leq f(3) < 1.
\end{align*}
so that $n \in \mc{L}_f$ in this case. \\
\underline{Case 2.2:} $2 \notin S_-$ or $2 \in S_-, 3 \notin S_-$\\
Here we take $j = -1$, $3 \leq p_1 < p_2$. Then $3||(n-1)$, which is compatible with $p_1||(n-1)$ whether $p_1 \geq 3$ (i.e., $2 \notin S_-$) or $p_1 = 2$ and $p_2>3$ (i.e., when $2 \in S_-$ and $3 \notin S_-$).
As above, since $4||(n+1)$, $n$ must be squarefree and coprime to $6$, and is necessarily also coprime to $p_1,p_2$ since $p_1||(n-1)$ and $p_2||(n+1)$. Thus, $n$ has no prime factors from $S_-$. Therefore, $f(n) = 1$.  On the other hand, $n-1$ and $(n+1)/4$ are both squarefree, and since $S_+$ is empty we have that $f((n-1)/[3,p_1]) \leq 1$ and $f((n+1)/(4p_2)) \leq 1$. In addition, $f(4) \leq 1$ as well, so since $(n-1)/p_1$ and $p_1$ are coprime we obtain 
$$
f(n-1) = f([3,p_1])f((n-1)/[3,p_1]) \leq f([3,p_1]) < 1.
$$
Similarly, as $(n+1)/(4p_2)$ is coprime to $4p_2$, and $p_2>2$,
$$
f(n+1) = f(4)f(p_2) f((n+1)/(4p_2)) \leq f(p_2) < 1.
$$
It follows that $n \in \mc{L}_f$ in this case as well, and the proof is complete.
\end{proof}

\subsection{When \eqref{eq:densefnot1} does hold}
We may now assume that \eqref{eq:densefnot1} does hold. Here, we use ``clipping'' in a different way than what was done in Section \ref{sec:RedtoFinConc}, to show the following. 
\begin{lem} \label{lem:clipLg}
Suppose there exists $P$ with $\Omega(P) \leq 1$ and that $\sup_{(n,P) = 1} |g_{\lambda}(n)| < \infty$. Then there exists an additive function $h$ with $\sup_n |h_{\lambda}(n)| < \infty$ and $\mc{L}_h \subseteq \mc{L}_g$.
\end{lem}
\begin{proof}
If $P = 1$ then the claim is immediate from taking $h = g$. Thus, we may assume in the sequel that $P = p_0$. \\
Let 
$$
C := \sup_{p_0 \nmid n} |g_{\lambda}(n)| + 100 |\lambda|,
\quad K := 2C + 1.
$$ 
Letting $\phi_K(t) := \text{sign}(t) \min\{K,|t|\}$ as before, we define the additive function $h$ on prime powers via 
$$
h(p^k) := \begin{cases} g(p^k) &\text{ if } p \neq p_0 \\ \lambda \log p_0^k + \phi_K(g_{\lambda}(p_0^k)) &\text{ if } p = p_0. \end{cases}
$$
Thus, we clearly have 
$$
\sup_n |h_{\lambda}(n)| \leq \sup_{p_0 \nmid n} |g_{\lambda}(n)| + \sup_{k \geq 1} |h_{\lambda}(p_0^k)| \leq C + K < \infty.
$$
Next, let $n \in \mc{L}_h$, and consider first that $p_0^\nu||n$ for some $\nu \geq 1$. Then $g(n\pm 1) = h(n\pm 1)$ and $g(n/p_0^{\nu}) = h(n/p_0^{\nu})$. Therefore, we have
$$
h(p_0^\nu) + g(n/p_0^\nu) > \max\{g(n-1),g(n+1)\}.
$$
Subtracting $\lambda \log n$ from both sides and rearranging, we equivalently have
\begin{equation}\label{eq:p_0nCase}
h_{\lambda}(p_0^\nu) > \max\{g_{\lambda}(n-1) + \lambda\log(1-1/n) , g_{\lambda}(n+1) + \lambda \log(1+1/n)\} - g_{\lambda}(n/p_0^\nu).
\end{equation}
In particular, uniformly over $n$, there is an absolute constant $c \in (0,100)$ such that
\begin{equation}\label{eq:withc}
h_{\lambda}(p_0^{\nu}) > \max\{g_{\lambda}(n-1) - g_{\lambda}(n/p_0^\nu), g_{\lambda}(n+1)-g_{\lambda}(n/p_0^\nu)\} - c|\lambda|. 
\end{equation}
Note that $h_{\lambda}(p_0^{\nu})$ has the same sign as $g_{\lambda}(p_0^\nu)$, and if the latter is $> K$ in absolute value then $|h_{\lambda}(p_0^\nu)| = K$. Consider the following (exhaustive) subcases: 
\begin{enumerate}[(a)]
\item if the RHS of \eqref{eq:p_0nCase} is positive then so is $h_{\lambda}(p_0^\nu)$. As $\phi_K(t) \leq t$ for all $t \geq 0$, 
we can replace $h$ by $g$ in \eqref{eq:p_0nCase}. In particular, adding $\lambda \log n$ back and rearranging, we get that
$$
g(n) > \max\{g(n-1),g(n+1)\},
$$
i.e., $n \in \mc{L}_g$.
\item If the RHS of \eqref{eq:p_0nCase} is negative and the LHS is non-negative then replacing $h_{\lambda}(p_0^\nu)$ by $g_{\lambda}(p_0^\nu)$ maintains the non-negativity of the LHS, and the inequality still holds. We obtain $n \in \mc{L}_g$ as in (a).
\item Finally, if the RHS of \eqref{eq:p_0nCase} is negative and the LHS is also negative, then because $c < 100$ the RHS in \eqref{eq:withc} has absolute value $< 2C$. It follows on multiplying both sides of \eqref{eq:p_0nCase} by $-1$ that
$$
0 < |h_{\lambda}(p_0^\nu)| < 2C < K
$$
Since $\phi_K(t) = t$ whenever $|t| \leq K$, we find that $g_{\lambda}(p_0^\nu) = h_{\lambda}(p_0^\nu)$, and so \eqref{eq:p_0nCase} continues to hold with $g$ replaced by $h$. 
\end{enumerate}
Thus, whenever $p_0 | n$ and $n \in \mc{L}_h$, $n \in \mc{L}_g$ as well. \\
Next, we suppose $p_0|(n^2-1)$. Without loss of generality we may assume that $p_0|(n+1)$; the proof is the same if $p_0|(n-1)$ instead (in the case where $p_0 = 2$, we may similarly only focus on $n+1$). In this case, in place of \eqref{eq:p_0nCase}, whenever $n \in \mc{L}_h$ we have
\begin{equation*}
g_{\lambda}(n) > g_{\lambda}((n+1)/p_0^\nu) + h_{\lambda}(p_0^\nu) + \lambda \log(1+1/n).
\end{equation*}
Rearranging, we then get that
\begin{align}\label{eq:p_0NotnCase}
h_{\lambda}(p_0^\nu) < g_{\lambda}(n) - g_{\lambda}((n+1)/p_0^\nu) - \lambda\log(1+1/n).
\end{align}
We do similar case work as above. 
\begin{enumerate}[(a)]
\item If the RHS of \eqref{eq:p_0NotnCase} is positive and the LHS is non-positive then replacing $h$ by $g$ preserves the inequality.
\item If the RHS of \eqref{eq:p_0NotnCase} is positive and the LHS is positive then (again since $c < 100$) we have
$$
0 < h_{\lambda}(p_0^\nu) < 2C < K,
$$
and so $g_{\lambda}(p_0^\nu) = h_{\lambda}(p_0^\nu)$ in \eqref{eq:p_0NotnCase}.
\item If the RHS of \eqref{eq:p_0NotnCase} is negative then the LHS is negative, and on multiplying both sides by $-1$, we have
$$
|h_{\lambda}(p_0^\nu)| > |g_{\lambda}(n) - g_{\lambda}((n+1)/p_0^\nu)- \lambda\log(1+1/n)|.
$$
Since $h_{\lambda}$ and $g_{\lambda}$ share the same sign and $|h_{\lambda}| \leq |g_{\lambda}|$ by Lemma \ref{lem:clipProp}, the inequality \eqref{eq:p_0NotnCase} continues to hold with $h$ replaced by $g$.
\end{enumerate}
We deduce that when $p_0|(n+1)$, if $n \in \mc{L}_h$ then $n \in \mc{L}_g$. The claim with $n-1$ in place of $n+1$ is the same. The proof is now complete.
\end{proof}

\begin{proof}[Proof of Proposition \ref{prop:omega1}]
Assume that (i) fails, so that $\delta(\mc{L}_f)\delta(\mc{S}_f) = 0$. By \eqref{eq:ftoreciprocal}, replacing $f$ by $1/f$ as before, we may assume that $\delta(\mc{L}_f) = 0$. By Lemma \ref{lem:clipLg} we may replace $f$ by $\tilde{f} = e^h$, wherein $\delta(\mc{L}_{\tilde{f}}) = 0$ and $\sup_n |h_{\lambda}(n)| < \infty$.\\
If $\tilde{f}(p) = 1$ for all but at most one prime $p$ then (ii) is verified with $c = 0$. Thus, we may assume otherwise that there are at least two primes $p_1,p_2$ with $\tilde{f}(p_j) \neq 1$. By Lemma \ref{lem:fequal1} we have $\sum_{p : \tilde{f}(p) \neq 1} 1/p = \infty$, otherwise $\mc{L}_{\tilde{f}}$ would have positive density, contradicting $\delta(\mc{L}_{\tilde{f}}) = 0$, as just established. Then, by Theorem \ref{thm:consec} the set of $n$ such that either $\tilde{f}(n) = \tilde{f}(n+1)$ or $\tilde{f}(n) = \tilde{f}(n-1)$ has logarithmic density $0$. Let $(m_N)_N$ denote the sequence of $n$ for which either $\tilde{f}(n) \in \{\tilde{f}(n+1), \tilde{f}(n-1)\}$, or $n \in \mc{L}_{\tilde{f}}$. Let $N$ be large and define $x_N := m_N-1$. Then
$$
\{m_1,\ldots,m_{N-1}\} = \mc{M}_h(x_N;1)
$$
in the notation of Lemma \ref{lem:valleys}. Setting $m_0 := 0$, as in the proof of Lemma \ref{lem:MRclipped}, we can partition $[0,x]$ as a union of intervals
$$
\bigcup_{0\leq j \leq N-1} [m_j,\ell_j) \cup [\ell_j,m_{j+1}),
$$
where $\tilde{f}$ (and thus $h$) is decreasing on $(m_j,\ell_j)$ and increasing on $[\ell_j,m_{j+1})$, for each $j$.
Therefore, since $\mc{M}_h(1)$ has logarithmic density $0$, we see that
\begin{align*}
&\frac{1}{\log x_N} \sum_{n \leq x_N} \frac{|h_{\lambda}(n+1)-h_{\lambda}(n)|}{n} = o_{N \ra \infty}(1) + \frac{1}{\log x_N} \sum_{n \leq x_N} \frac{|h(n+1)-h(n)|}{n} \\
&\leq o_{N \ra \infty}(1) + \frac{1}{\log x_N} \sum_{1 \leq j\leq N-1} \left(\sum_{m_j < n < \ell_j} \frac{h(n)-h(n+1)}{n} + \sum_{\ell_j \leq n < m_{j+1}} \frac{h(n+1)-h(n)}{n}\right) \\
&\leq o_{N \ra \infty}(1) + \frac{1}{\log x_N} \sum_{1 \leq j\leq N-1} \frac{1}{m_j}\left(\sum_{m_j < n < \ell_j} (h_\lambda(n)-h_\lambda(n+1)) + \sum_{\ell_j \leq n < m_{j+1}} (h_\lambda(n+1)-h_\lambda(n))\right) \\
&\leq o_{N \ra \infty}(1) + \frac{1}{\log x_N}\sum_{1 \leq j \leq N-1} \frac{1}{m_j} (h_{\lambda}(m_j+1) + h_{\lambda}(m_{j+1}) - 2h_{\lambda}(\ell_j)).
\end{align*}
Combining $\sup_{n \geq 1} |h_{\lambda}(n)| < \infty$ with the fact that $\mc{M}_h(1)$ has logarithmic density $0$ by Theorem \ref{thm:consec} and the assumption $\delta(\mc{L}_f)$, it follows that
\begin{align*}
\frac{1}{\log x_N} \sum_{n \leq x_N} \frac{|h_{\lambda}(n+1)-h_{\lambda}(n)|}{n} 
= o_{N \ra \infty}(1).
\end{align*}
By Theorem \ref{thm:KataiLog}, we see that $h_{\lambda}(n) = c\log n$ for some $c$, and since $h_{\lambda}$ is bounded we get $c = 0$. We therefore conclude that $g(p) = \lambda \log p$ at all but at most the single prime $p = p_0$,
and the claim follows.  
\end{proof}

\begin{appendix}
\section{Appendix: Proofs of the key lemmas from \cite{ErdAI}}
Given that the paper \cite{ErdAI} has been left unpublished (and moreover does not appear on arXiv or a comparably reliable repository), we will in this appendix reproduce (with some elaboration) the proofs of the main lemmas quoted throughout this paper. All proofs have been written by the author, and in several cases the arguments differ from those in \cite{ErdAI}.\\
For the convenience of the reader, we restate the relevant claims here. Recall that $g: \mb{N} \ra \mb{R}$ is an additive function, $K \geq 1$ and $M \in (0,1/4)$ are fixed parameters, and $\phi_K(t) := \text{sign}(t) \min\{|t|,K\}$. For each $x$ we define parameters $\lambda_x \in \mb{R}$ and $s_x > 0$ such that
$$
M := \frac{1}{s_x^2}\left(\lambda_x^2 + \sum_{p \leq x} \frac{\min\{s_x^2,g_{\lambda_x}(p)^2\}}{p}\right)
$$
The result below, which is \cite[Lem. 4.1]{ErdAI}, gives the existence of $s_x$ and $\lambda_x$ for each $M$, as well as some of their properties of relevance in the sequel.
\begin{lem}[cf. Lem. 4.1 of \cite{ErdAI}]\label{lem:4point1}
Let $g$ be a real-valued additive function that satisfies
\begin{equation}\label{eq:Vtoinf}
V_g(x) = \inf_{\lambda \in \mb{R}} \sum_{p \leq x} \frac{\min\{1,(g(p)-\lambda \log p)^2\}}{p} \ra \infty \text{ as } x \ra \infty.
\end{equation}
For each fixed $M > 0$ and all sufficiently large $X$ there are $s_X > 0$ and $c_X \in \mb{R}$ such that 
\begin{equation}\label{eq:McXsXDef}
M = c_X^2 + \sum_{p \leq X} \frac{\min\{1,(g(p)/s_X - c_X\log p)^2\}}{p}.
\end{equation}
Furthermore, setting $\lambda_X := c_Xs_X$, the following statements hold:
\begin{enumerate}[(a)]
\item As $X \ra \infty$, $s_X \ra \infty$ and $c_X \ra 0$. 
\item We have
$$
\sum_{\ss{p \leq X \\ |g_{\lambda_X}(p)| \leq s_X}} \frac{g_{\lambda_X}(p)\log p}{p} = \lambda_X.
$$
\end{enumerate}
\end{lem}
\begin{proof}
For large $X$ and $c \in \mb{R}$, $s > 0$ define
$$
W_X(c,s) := c^2 + \sum_{p \leq X} \frac{\min\{1,(g(p)/s - c \log p)^2\}}{p}, \ W_X(s) := \min_{c \in \mb{R}} W_X(c,s).
$$
Observe first of all that $W_X(s)$ is a non-decreasing function of $X$. Indeed, if $X' > X$, and we let $c_{X'}, c_X$ be chosen with $W_X(s) = W_X(c_X,s)$ and $W_{X'}(s) = W_{X'}(c_{X'},s)$, we have on restricting to the primes $p \leq X$ that
$$
W_{X'}(s) = W_{X'}(c_{X'},s) \geq c_{X'}^2 + \sum_{p \leq X} \frac{\min\{1,(g(p)/s - c_{X'}\log p)^2\}}{p} = W_X(c_{X'},s) \geq W_X(s).
$$
Next, we show that there is a solution $s_X$ to $W_X(s_X) = M$.
First of all, suppose that $s > 0$ is fixed.  Note that by considering the cases $s \geq 1$ and $s < 1$ separately, it is easily shown that
$$
\min\{s^2,g_{\lambda}(p)^2\} \geq \min\{1,s^2\} \min\{1,g_{\lambda}(p)^2\}.
$$
From \eqref{eq:Vtoinf},  as $X \ra \infty$ we obtain
\begin{align*}
W_X(s) &= \min_{c \in \mb{R}} \left(c^2 + \frac{1}{s^2} \sum_{p \leq X} \frac{\min\{s^2,g_{cs}(p)^2\}}{p}\right) \geq \frac{1}{s^2} \min_{\lambda \in \mb{R}} \left(\lambda^2 + \min\{1,s^2\} \sum_{p \leq X} \frac{\min\{1,g_{\lambda}(p)^2\}}{p}\right)  \\
&\geq \frac{\min\{1,s^2\}}{s^2} V_g(X) \ra \infty.
\end{align*}
In particular, for each $s$ there exists $X_0(s)$ such that if $X \geq X_0(s)$ then $W_X(s) > M$. \\
On the other hand, note that since $\min\{1,a\} \leq \min\{1,b\}$ whenever $0 \leq a \leq b$, we find that if $s_2 \geq s_1$ then for fixed $X$ and any $\lambda \in \mb{R}$,
\begin{align*}
\frac{\lambda^2}{s_2^2} + \sum_{p \leq X} \frac{\min\{1,g_{\lambda}(p)^2/s_2^2\}}{p} \leq \frac{\lambda^2}{s_1^2} + \sum_{p \leq X} \frac{\min\{1,g_{\lambda}(p)^2/s_1^2\}}{p}.
\end{align*}
Taking $\lambda = c_Xs_1$, we easily deduce that $W_X(s_2) \leq W_X(s_1)$. Furthermore, keeping $X$ fixed, as $s \ra \infty$ we have (on taking $\lambda = 0$ in an upper bound),
$$
W_X(s) \leq \frac{1}{s^2} \sum_{p \leq X} \frac{g(p)^2}{p} \ra 0.
$$
Therefore, whenever $s_0$ is fixed and $X \geq X_0(s_0)$, there is $s_1 > s_0$ such that
$$
W_X(s_1) < M < W_X(s_0).
$$
Clearly, $W_X(s)$ is a continuous function of $s$, so by the intermediate value theorem, we deduce that there is $s' \in (s_0,s_1)$ with $W_X(s') = M$, and we choose $s_X$ to be the largest value of $s'$ for which this holds. Choosing a minimiser $c_X$ for $W_X(s_X)$ gives the required parameter $c_X$. \\
We next show the properties (a) and (b). For (a), we observe that whenever $X' > X$ we have $s_{X'} \geq s_X$. Indeed, if not then on the one hand, since $W_X(s)$ is increasing in the $X$ aspect, we have
$$
W_{X'}(s_{X}) \geq W_X(s_X) = M.
$$
On the other hand, if $s_{X'} < s_X$ then as $W_X(s)$ is decreasing in the $s$-aspect,
$$
W_{X'}(s_X) \leq W_{X'}(s_{X'}) = M,
$$
which either gives the contradiction $M < M$ if $W_{X'}(s_X) < W_{X'}(s_{X'})$, or else if $W_{X'}(s_{X}) = M$ we obtain the contradiction that $s_{X'}$ is not the maximal value of $s$ with $W_{X'}(s) = M$. In any case, we see that $X \mapsto s_X$ is a non-decreasing function. If $s_X \ra S < \infty$ as $X \ra \infty$ then $s_X \leq S$. But then $M = W_X(s_X) \geq W_X(S)$ for all $X$, which contradicts the fact we showed above that $W_X(S) \ra \infty$ as $X \ra \infty$. We thus deduce that $s_X \ra \infty$. \\
Next, we show that $c_X \ra 0$. Observe first of all that $c_X^2 \leq W_X(c_X,s_X) = M$, so $|c_X| \leq \sqrt{M}$ for all $X$. Now, if a subsequence $(c_{X_j})_j$ converged to some $c \neq 0$ then by Fatou's lemma, for every large $Y$ we would have
$$
\sum_{p \leq Y} \frac{\min\{1,(c\log p)^2\}}{p} = \sum_{p \leq Y} \lim_{j \ra \infty} \frac{\min\{(f(p)/s_{X_j} - c_{X_j}\log p)^2,1\}}{p} \leq \lim_{j \ra \infty} W_{X_j}(c_{X_j},s_{X_j}) = M.
$$
But the LHS is at least
$$
\sum_{e^{1/c} < p \leq Y} \frac{1}{p} = \log\log Y - \log(1/c) \ra \infty, 
$$
as $Y \ra \infty$, which is a contradiction. This proves that every subsequence of $(c_X)_X$ converges to $0$, thus $c_X \ra 0$ as well, finishing the proof of (a). \\
To prove (b), taking $\lambda_X = c_Xs_X$ as in the statement, we define the auxiliary function
$$
\tilde{W}_X(c) := c^2 + \sum_{\ss{p \leq X \\ |g_{\lambda_X}(p)| > s_X}} \frac{1}{p} + \sum_{\ss{p \leq X \\ |g_{\lambda_X}(p)| \leq s_X}} \frac{(g(p)/s_X - c\log p)^2}{p}.
$$
Note that this is differentiable in $c$, and satisfies $\tilde{W}_X(c_X) = W_X(c_X,s_X)=M$. Furthermore, we observe that since $c_X$ is a minimiser for $c\mapsto W_X(c,s_X)$,
\begin{align*}
\tilde{W}_X(c) &\geq c^2 + \sum_{\ss{p \leq X \\ |g_{\lambda_X}(p)| > s_X}} \frac{\min\{1,(g(p)/s_X-c\log p)^2\}}{p} + \sum_{\ss{p \leq X \\ |g_{\lambda_X}(p)| \leq s_X}} \frac{\min\{1,(g(p)/s_X-c\log p)^2\}}{p} \\
&= W_X(c,s_X) \geq W_X(c_X,s_X),
\end{align*}
and thus $c = c_X$ is a minimum for $\tilde{W}_X$ as well. It follows on differentiating that
$$
0 = \tilde{W}_X'(c_X) = 2c_X - \frac{2}{s_X}\sum_{\ss{p \leq X \\ |g_{\lambda_X}(p)| \leq s_X}} \frac{g_{\lambda_X}(p) \log p}{p},
$$
and rearranging this, we obtain
$$
\sum_{\ss{p \leq X \\ |g_{\lambda_X}(p)| \leq s_X}} \frac{g_{\lambda_X}(p) \log p}{p} = c_X s_X = \lambda_X,
$$
as claimed. \\
\end{proof}
Next, for each $x\geq 1$ we define 
$$
\tilde{g}_x(n) := \sum_{p|n} \phi_K(g_{\lambda_x}(p)/s_x), \quad \mu(x) := \sum_{p \leq x} \frac{\tilde{g}_x(p)}{p}.
$$
\begin{lem}[Lem. 5.2 of \cite{ErdAI}]\label{lem:LowBdApp}
There are absolute constants $c_0,C_0 > 0$ such that
$$
\liminf_{X \ra \infty} \frac{1}{X}\sum_{n \leq X} |\tilde{g}_X(n) - \mu(X)|^2 \geq c_0M - C_0(KM)^2.
$$
\end{lem}
\begin{proof}
For an additive function $h$, define
$$
A_h(x) := \sum_{p^k \leq x} \frac{h(p^k)}{p^k} \left(1-\frac{1}{p}\right).
$$
By Ruzsa's variance lemma \cite{RuzsaVar},
\begin{align}\label{eq:RuzsaVar}
\frac{1}{x} \sum_{n \leq x} |\tilde{g}_x(n) - A_{\tilde{g}_x}(x)|^2 \asymp \min_{\theta \in \mb{R}} \left(\theta^2 + \sum_{p^k \leq x} \frac{(\tilde{g}_x(p^k)-\theta \log p^k)^2}{p^k}\right) \geq \min_{\theta \in \mb{R}} \sum_{p \leq x} \frac{(\tilde{g}_x(p)-\theta \log p)^2}{p}.
\end{align}
Since $\tilde{g}_x$ is strongly additive and $|\tilde{g}_x(p)| \leq K$ for all $p$,
\begin{equation}\label{eq:AmuComp}
A_{\tilde{g}_x}(x) = \sum_{p\leq x} \tilde{g}_x(p) \sum_{1 \leq k \leq \lfloor \tfrac{\log x}{\log p}\rfloor} \left(\frac{1}{p^k}-\frac{1}{p^{k+1}}\right) = \sum_{p \leq x} \frac{\tilde{g}_x(p)}{p} + O\left(\frac{1}{x} \sum_{p \leq x} |\tilde{g}_x(p)|\right) = \mu(x) + O\left(\frac{K}{\log x}\right).
\end{equation}
We thus obtain
\begin{align} \label{eq:redtheta}
\frac{1}{x}\sum_{n \leq x} |\tilde{g}_x(n) - \mu(x)|^2 \gg \min_{\theta \in \mb{R}} \sum_{p \leq x} \frac{(\tilde{g}_x(p)-\theta \log p)^2}{p} + o(1).
\end{align}
Now, differentiating the expression in the RHS in terms of $\theta$ shows that the minimum $\theta = \theta_0$ satisfies
$$
\theta_0 = \left(\sum_{p \leq x} \frac{\tilde{g}_x(p) \log p}{p}\right)/\left(\sum_{p \leq x} \frac{(\log p)^2}{p}\right) = \left(1+o(1)\right) \frac{2}{(\log x)^2} \sum_{p\leq x} \frac{\tilde{g}_x(p)\log p}{p},
$$
using partial summation and the prime number theorem in the last step. Since $K \geq 1$, by Lemma \ref{lem:4point1}(b) we have
\begin{align*}
\left|\sum_{p \leq x} \frac{\tilde{g}_x(p)\log p}{p}\right| &\leq \left|\sum_{\ss{p \leq x \\ |g_{\lambda_x}(p)| \leq s_x}} \frac{(g_{\lambda_x}(p)/s_x) \log p}{p}\right| + K\sum_{\ss{p \leq x \\ |g_{\lambda_x}(p)| > s_x}} \frac{\log p}{p} \\
&\leq \frac{|\lambda_x|}{s_x} + K(\log x) \sum_{\ss{p \leq x \\ |g_{\lambda_x}(p)| > s_x}} \frac{1}{p} \leq |c_x| + KM\log x.
\end{align*}
In particular, we obtain that
$$
\left|\theta_0\sum_{p \leq x} \frac{\tilde{g}_x(p)\log p}{p}\right| \leq (4+o(1))\left(\left(\frac{|c_x|}{\log x}\right)^2 + (KM)^2\right).
$$
Furthermore, again because $K \geq 1$, \eqref{eq:McXsXDef} yields
\begin{align*}
\sum_{p \leq x} \frac{\tilde{g}_x(p)^2}{p} \geq \sum_{p \leq x} \frac{\min\{(g_{\lambda_x}(p)/s_x)^2,1\}}{p} = M - c_x^2.
\end{align*}
Since by Lemma \ref{lem:4point1}(a), $c_x \ra 0$ as $x \ra \infty$, we obtain from \eqref{eq:redtheta} that there are constants $C,C_0 > 0$ such that
\begin{align*}
\frac{1}{x} \sum_{n \leq x} (\tilde{g}_x(n) - \mu(x))^2 \geq C \left(\sum_{p \leq x} \frac{\tilde{g}_x(p)^2}{p} - 2\theta_0 \sum_{p \leq x} \frac{\tilde{g}_x(p)\log p}{p}\right) \geq c_0 M - C_0 (KM)^2 + o(1).
\end{align*}
The claim follows on taking $x \ra \infty$ along any subsequence.
\end{proof}
In the sequel, we will appeal several times to the following fourth-moment estimate, due to Elliott \cite[Thm. 1, $\beta = 4$]{Elliott}.
\begin{lem}[Elliott] \label{lem:Elliott}
Let $h$ be an additive function. Then
$$
\frac{1}{x} \sum_{n \leq x} |h(n) - A_h(x)|^4 \ll \left(\sum_{p^k \leq x} \frac{|h(p^k)|^2}{p^k}\right)^2 + \sum_{p^k \leq x} \frac{|h(p^k)|^4}{p^k}.
$$
\end{lem}
As above, for an arithmetic function $F: \mb{N} \ra \mb{C}$ and integers $n > H\geq 1$, define 
$$
(\mc{A}_H F)(n) := \frac{1}{H} \sum_{1 \leq j \leq H} F(n-j).
$$
\begin{lem}[cf. Lem. 2.1 of \cite{ErdAI}] \label{lem:2point1}
Let $H = H(x) \ra \infty$ be an integer parameter with $H \leq \sqrt{x}$. Then as $x \ra \infty$,
$$
\frac{1}{x}\sum_{H < n \leq x} \left(\mc{A}_H\tilde{g}_x(n) - \mu(x)\right)^2 = o(1).
$$
\end{lem}
\begin{proof}
We may assume that $x$ is sufficiently large so that $H(x) \geq 10$ and $H(x) \leq \sqrt{x} \leq x/(100\log x)$, which will be relevant shortly. \\
For $m \in \mb{N}$ let us write
$$
S_m := \frac{1}{x} \sum_{H < n \leq x} \left|\mc{A}_H\tilde{g}_x(n) - \mu(x)\right|^m.
$$
The LHS in the statement of the lemma is of course $S_2$. By H\"{o}lder's inequality, we have
\begin{equation}\label{eq:S2S1S4}
S_2 \leq S_1^{2/3} S_4^{1/3}.
\end{equation}
We first bound $S_4$. Applying H\"{o}lder's inequality and \eqref{eq:AmuComp},
\begin{align*}
S_4 = \frac{1}{x} \sum_{H < n \leq x } |\mc{A}_H(\tilde{g}_x(n)-\mu(x))|^4 &\leq \frac{1}{H} \sum_{1 \leq j \leq H} \sum_{H < n \leq x} |\tilde{g}_x(n-j) - \mu(x)|^4 \\
&\ll \frac{1}{x} \sum_{n \leq x} |\tilde{g}_x(n) - A_{\tilde{g}_x}(x)|^4 + O_K\left(\frac{1}{(\log x)^4}\right). 
\end{align*}
Since $\tilde{g}_x$ is strongly additive and $|\tilde{g}_x(p)| \leq K$, Lemma \ref{lem:Elliott} yields
\begin{align*}
\frac{1}{x} \sum_{n \leq x} |\tilde{g}_x(n) - A_{\tilde{g}_x}(x)|^4 \ll \left(\sum_{p^k \leq x} \frac{\tilde{g}_x(p^k)^2}{p^k}\right)^2 + \sum_{p^k \leq x} \frac{\tilde{g}_x(p^k)^4}{p^k} \ll \left(\sum_{p \leq x} \frac{\tilde{g}_x(p)^2}{p}\right)^2 + K^2\sum_{p \leq x} \frac{\tilde{g}_x(p)^2}{p}.
\end{align*}
Since $K \geq 1$, we clearly have
\begin{align} \label{eq:vargxBd}
\sum_{p \leq x} \frac{\tilde{g}_x(p)^2}{p} = \sum_{p \leq x} \frac{\min\{K^2,(\tilde{g}_{\lambda_x}(p)/s_x)^2\}}{p} \leq K^2M.
\end{align}
As $M < 1$, it follows that
\begin{align*}
S_4 \ll (K^2 M)^2 + K^4M + o(1) \ll K^4M + o(1).
\end{align*}
To estimate $S_1$, we recall first of all our assumption that $10 \leq H \leq x/(100 \log x)$. We decompose the range dyadically and replace $\mu(x)$ by $\tilde{A}_{g_x}(x)$ in the range $(H,x/\log x]$, getting
\begin{align*}
S_1 &= \frac{1}{x}\sum_{H<n \leq x/\log x} |\mc{A}_H\tilde{g}_x(n) - A_{\tilde{g}_x}(x)| + \sum_{1 \leq 2^j \leq \log x} \frac{1}{2^{j+1}} \cdot \frac{2^{j+1}}{x} \sum_{x/2^{j+1} < n \leq x/2^j} |\mc{A}_H\tilde{g}_x(n) - \mu(x)|  + o(1)\\
&=: S_{1,1} + S_{1,2}.
\end{align*}
By the triangle inequality, the Tur\'{a}n-Kubilius inequality (or the first estimate in \eqref{eq:RuzsaVar}, obtaining an upper bound by setting $\theta = 0$) and \eqref{eq:vargxBd}, the range $(H,x/\log x]$ contributes 
\begin{align*}
S_{1,1} &\leq \frac{1}{H} \sum_{1 \leq j \leq H} \frac{1}{x} \sum_{H < n \leq x/\log x} |\tilde{g}_x(n-j) - A_{\tilde{g}_x}(x)| \leq \frac{1}{x} \sum_{n \leq x/\log x} |\tilde{g}_x(n) - A_{\tilde{g}_x}(x)| \\
&\ll \frac{1}{\sqrt{\log x}} \left(\frac{1}{x}\sum_{n \leq x} (\tilde{g}_x(n) - A_{\tilde{g}_x}(x))^2\right)^{1/2} \ll \frac{1}{\sqrt{\log x}} \left(\sum_{p \leq x} \frac{\tilde{g}_x(p)^2}{p}\right)^{1/2} \ll \frac{KM^{1/2}}{\sqrt{\log x}},
\end{align*}
so that $S_{1,1} = o(1)$. Now, for each $j$ with $1 \leq 2^j \leq \log x$ we have by Mertens' theorem,
$$
|\mu(x) - \mu(x/2^j)| \leq \sum_{x/\log x < p \leq x} \frac{|\tilde{g}_x(p)|}{p} \ll K \frac{\log\log x}{\log x}.
$$
Furthermore, for $x/\log x < y \leq x$ we observe the long sum bound
\begin{align*}
\frac{2}{y}\sum_{y/2 < n \leq y} \tilde{g}_x(n) = \frac{2}{y} \sum_{p\leq y} \tilde{g}_x(p) \sum_{\ss{y/2 < n \leq y \\ p|n}} 1 = \mu(y) + O\left(\frac{1}{y}\sum_{p \leq y} |\tilde{g}_x(p)|\right) = \mu(y) + O\left(\frac{K}{\log y}\right).
\end{align*}
Since $10 \leq H \leq y/100$ for all $y$ in this range (and $H$ is an integer by assumption), we may appeal to a result of the author \cite[Thm. 1.1]{Mangerel} that, together with\eqref{eq:vargxBd}, yields
\begin{align*}
\frac{2}{y} \sum_{y/2 < n \leq y} |\mc{A}_H\tilde{g}_x(n) - \mu(x)| &\leq \frac{2}{y} \sum_{y/2 < n \leq y} \left|\mc{A}_H\tilde{g}_x(n) - \frac{2}{y} \sum_{y/2< n \leq y} \tilde{g}_x(n)\right| + O\left(\frac{K\log\log x}{\log x}\right) \\
&\ll \left(\left(\frac{\log\log H}{\log H}\right)^{1/2} + (\log y)^{-1/800}\right) \left(\sum_{p \leq x} \frac{\tilde{g}_x(p)^2}{p}\right)^{1/2} \\
&\ll \left(\left(\frac{\log\log H}{\log H}\right)^{1/2} + (\log x)^{-1/800}\right) KM^{1/2}.
\end{align*}
Applying this with $y=x/2^j$, then summing over $1 \leq 2^j \leq \log x$, we obtain that $S_{1,2} = o(1)$. Therefore, $S_1 = o(1)$ as well.  \\
Combining this with the previous bound for $S_4$ and inserting both into \eqref{eq:S2S1S4}, the claim follows.
\end{proof}
Next, we recall the notation
$$
h_x(n) := \phi_K(g_{\lambda_x}(n)/s_x - a_x), \quad a_x := \frac{1}{s_x} \sum_{\ss{p \leq x \\ |g_{\lambda_x}(p)| \leq s_x}} \frac{g_{\lambda_x}(p)}{p}.
$$
Set also $G_x(n) := \tilde{g}_x(n) - a_x$.
\begin{lem}[cf. Lem. 5.1 of \cite{ErdAI}]\label{lem:5point1}
There is a constant $C' > 0$ such that as $x \ra \infty$,
\begin{align*}
\frac{1}{x} \sum_{n \leq x} |h_x(n) - G_x(n)|^2 \leq C' \left(\frac{M}{K^2} + M^2\log(1/M) + (KM)^2\right) + o(1).
\end{align*}
\end{lem}
\begin{proof}
In the sequel, we define the auxiliary function $h_x^\ast(n)$ as follows:
$$
h_x^\ast(n) := \phi_K\left(\sum_{p|n} g_{\lambda_x}(p)/s_x - a_x\right),
$$
which differs from $h_x(n)$ in that it does not depend on the prime power factors $p^k||n$ with $k \geq 2$ (and, in this respect, is more comparable to $G_x(n)$). Now, as before, we set
$$
\mc{T}_x := \{p \leq x : \ |g_{\lambda_x}(p)| > s_x\}.
$$
Let us write $(n,\mc{T}_x) = 1$ to denote that $p |n \Rightarrow p \notin \mc{T}_x$. 
Note that when $M$ is small, $\mc{T}_x$ is a relatively sparse set, and thus $(n,\mc{T}_x) > 1$ is an unlikely event.  Write also $\omega_{\mc{T}_x}(n)$ to denote the number of prime factors of $n$ coming from $\mc{T}_x$ (so that $\omega_{\mc{T}_x}(n) = 0$ is equivalent to $(n,\mc{T}_x) = 1$). \\
With the above additional notation, we have
\begin{align*}
\frac{1}{x}\sum_{n \leq x} |h_x(n) - G_x(n)|^2 &\ll \frac{1}{x} \sum_{\ss{n \leq x \\ (n,\mc{T}_x) = 1}} |h_x^\ast(n) - G_x(n)|^2 + \frac{1}{x}\sum_{\ss{n \leq x \\ \omega_{\mc{T}_x}(n) \geq 1}} |h_x^\ast(n) - G_x(n)|^2 + \frac{1}{x}\sum_{\ss{n \leq x}} |h_x(n)-h_x^\ast(n)|^2\\
&=: V_1 + V_2 + V_3
\end{align*}
We start with $V_1$. Since $K \geq 1$, whenever $(n,\mc{T}_x) = 1$ we have
$$
G_x(n) = \tilde{g}_x(n) - a_x = \sum_{p|n} g_{\lambda_x}(p)/s_x - a_x = h_x^\ast(n),
$$
unless $|g_{\lambda_x}(n)/s_x - a_x| > K$. Moreover, when restricted to $n$ that satisfy $(n,\mc{T}_x) = 1$, $a_x$ is precisely the mean value of $\tilde{g}_x(n)$. It follows that
\begin{align*}
V_1 &\leq 2\sum_{\ss{n \leq x \\ (n,\mc{T}_x) = 1 \\ |g_{\lambda_x}(n)/s_x -a_x| > K}} \left(K^2 + |g_{\lambda_x}(n)/s_x - a_x|^2\right) \ll \frac{1}{K^2} \sum_{\ss{n \leq x \\ (n,\mc{T}_x) = 1}} |g_{\lambda_x}(n)/s_x - a_x|^4.
\end{align*}
In view of \eqref{eq:AmuComp}, we note that 
$$
a_x = \mu(x) - \frac{1}{s_x}\sum_{\ss{p \leq x \\ 1 < |g_{\lambda_x}(p)|/s_x \leq K}} \frac{g_{\lambda_x}(p)}{p} = \mu(x) + O(KM) = A_{\tilde{g}_x}(x) + O\left(K\left(M + \frac{1}{\log x}\right)\right).
$$
Since $|\tilde{g}_x(p)| \leq 1$ whenever $p \notin\mc{T}_x$, together with Lemma \ref{lem:Elliott} this gives
\begin{align*}
&\frac{1}{K^2} \sum_{\ss{n \leq x \\ (n,\mc{T}_x) = 1}} |g_{\lambda_x}(n)/s_x-a_x|^4 \ll \frac{1}{K^2} \sum_{\ss{n \leq x \\ (n,\mc{T}_x) = 1}} |\tilde{g}_x(n) - A_{\tilde{g}_x}(x)|^4 + K^2M^4 + o(1) \\
&\ll \frac{1}{K^2} \left(\left(\sum_{p \leq x} \frac{\min\{1,(g_{\lambda_x}(p)/s_x)^2\}}{p}\right)^2 + \sum_{p \leq x} \frac{\min\{1,(g_{\lambda_x}(p)/s_x)^2\}^2}{p}\right) + K^2M^4 + o(1) \\
&\ll \frac{M^2 + M}{K^2} + K^2M^4 + o(1),
\end{align*}
which, since $M < 1$, is an acceptable bound. \\
Next, consider $V_3$, which only relates to the contribution of prime power factors to $h_x$. Fix some parameter $T = T(M)$ to be chosen shortly. We split the sum over $n$ into sets $\mc{G}_T$ and $\mc{B}_T$, where
$$
\mc{G}_T := \{n \leq x : \ p^k||n, k \geq 2 \Rightarrow p^k \leq T\}, \quad \mc{B}_T := [1,x] \bk \mc{G}_T.
$$
Since $|h_x(n) - h_x^\ast(n)| \leq 2K$ for all $n$, we see that
\begin{align*}
\frac{1}{x}\sum_{\ss{n \in \mc{B}_T}} |h_x(n)-h_x^\ast(n)|^2 \leq \frac{4K^2}{x} \sum_{\ss{p^k > T \\ k \geq 2}} \sum_{\ss{n \leq x \\ p^k|n}} 1 \ll K^2 \sum_{\ss{p^k > T \\ k \geq 2}} \frac{1}{p^k} \leq K^2 M^2,
\end{align*}
provided that $T$ is sufficiently large relative to $M$. We fix a corresponding $T$ for the remainder of the argument. \\
Next, we note that the set of prime powers $p^k \leq T$ with $k \geq 2$ depends at most on $M$, but not on $x$. Therefore, there is a constant $C = C_{f,M} > 0$ such that
$$
\frac{1}{s_x}\sum_{\ss{p^k \leq T \\ k \geq 2}} \frac{|g_{\lambda_x}(p^k)|}{p^k} \leq \frac{1}{s_x}\sum_{\ss{p^k \leq T \\ k \geq 2}} \frac{|g(p^k)|}{p^k} + \frac{\lambda_x}{s_x} \sum_{\ss{p^k \leq T \\ k \geq 2}} \frac{\log p^k}{p^k} \leq C\left(\frac{1}{s_x} + c_x\right).
$$
Therefore, increasing $C$ as needed, we have
$$
\frac{1}{x}\sum_{\ss{n \in \mc{G}_T}} |h_x(n) - h_x^\ast(n)|^2 \leq \frac{2K}{x} \sum_{n \leq x}\left(\frac{1}{s_x}\sum_{\ss{p^k || n \\ p^k \leq T, \ k \geq 2}} |g_{\lambda_x}(p^k)|\right) \leq \frac{2K}{s_x} \left(\sum_{\ss{p^k \leq T \\ k \geq 2}} \frac{|g_{\lambda_x}(p^k)|}{p^k} + O_{T,f}(1)\right) \leq  C\left(\frac{1}{s_x} + c_x\right).
$$
By Lemma \ref{lem:4point1}(a), $s_x \ra \infty$ and $c_x \ra 0$, therefore
\begin{align*}
\frac{1}{x}\sum_{n \in \mc{G}_T} |h_x(n)-h_x^\ast(n)|^2 = o(1).
\end{align*}
Combined with the contribution from $\mc{B}_T$, we obtain 
$$
V_3 \ll (KM)^2 + o(1).
$$
It remains to estimate $V_2$. Here, we consider the cases that $\omega_{\mc{T}_x}(n) = 1$ and $\omega_{\mc{T}_x}(n) \geq 2$ separately. \\
Write $n = tm$, where $p^k||n$, $p\in \mc{T}_x$ implies $p^k||t$. Write also $t^\ast := \prod_{p|t} p$ to be the squarefree kernel of $t$. Then as $K \geq 1$,
$$
h_x^\ast(n) - G_x(n) = \phi_K\left(\sum_{p|t} g_{\lambda_x}(p)/s_x + G_x(m)\right) - G_x(t) - G_x(m).
$$
If $\omega_{\mc{T}_x}(n) = \omega(t) = 1$ then $G_x(t) = \phi_K(g_{\lambda_x}(t^\ast)/s_x)$. By the Lipschitz property of $\phi_K$ in Lemma \ref{lem:clipProp}(i),
$$
|h_x^{\ast}(n)-G_x(n)| \leq |G_x(m)| + |\phi_K(g_{\lambda_x}(t^\ast)/s_x + G_x(m)) - \phi_K(g_{\lambda_x}(t^\ast)/s_x)| \leq 2|G_x(m)|.
$$
If, instead, $\omega_{\mc{T}_x}(n) \geq 2$ then we have $|G_x(t)| \leq \omega_{\mc{T}_x}(t)K$ and so
\begin{align*}
|h_x^\ast(n) - G_x(n)| &\leq |\phi_K(g_{\lambda_x}(t^\ast)/s_x + G_x(m)) - \phi_K(g_{\lambda_x}(t^\ast)/s_x)| + |G_x(m)| +|\phi_K(g_{\lambda_x}(t^\ast)/s_x)| +  |G_x(t)| \\
&\leq K(\omega_{\mc{T}_x}(n) + 1) + 2|G_x(m)| \leq 2K \omega_{\mc{T}_x}(n) + 2|G_x(m)|.
\end{align*}
It follows in either case that
\begin{align*}
V_2 &\ll \frac{K^2}{x} \sum_{n \leq x} \omega_{\mc{T}_x}(n)(\omega_{\mc{T}_x}(n)-1) + \sum_{\ss{1 < t \leq x \\ p|t\Rightarrow p \in \mc{T}_x}} \frac{1}{t} \cdot \frac{t}{x} \sum_{\ss{m \leq x/t \\ (m,\mc{T}_x) = 1}} |\tilde{g}_x(m)-a_x|^2 =: V_{2,1} + V_{2,2}. 
\end{align*}
The term $V_{2,1}$ is easily bounded via
\begin{align*}
V_{2,1} = \frac{1}{x}\sum_{\ss{pq \leq x \\ p,q \in \mc{T}}} \left\lfloor \frac{x}{[p,q]}\right\rfloor -\frac{1}{x}\sum_{\ss{p \leq x \\ p \in \mc{T}}} \left\lfloor \frac{x}{p}\right\rfloor \leq \sum_{\ss{p,q \leq x \\ p \neq q}} \frac{1}{pq} \ll M^2.
\end{align*}
Next, we consider $V_{2,2}$. First, for each $1 < t \leq x$ with $p|t \Rightarrow p \in \mc{T}_x$, write
$$
\tilde{a}_{x/t} := \frac{1}{s_x}\sum_{\ss{p \leq x/t \\ |g_{\lambda_x}(p)| \leq s_x}} \frac{g_{\lambda_x}(p)}{p}.
$$
By the Tur\'{a}n-Kubilius inequality and the constraint $|\tilde{g}_x(p)| \leq 1$ for all $p|n$ with $(n,\mc{T}_x) = 1$, we have 
\begin{align} \label{eq:TaylorExp}
&\sum_{\ss{1 < t \leq x \\ p|t\Rightarrow p \in \mc{T}_x}} \frac{1}{t} \cdot \frac{t}{x} \sum_{\ss{m \leq x/t \\ (m,\mc{T}_x)=1}} |\tilde{g}_x(m) - \tilde{a}_{x/t}|^2 \nonumber\\
&\ll \sum_{\ss{1 < t \leq x \\ p|t \Rightarrow p \in \mc{T}_x}} \frac{1}{t} \sum_{\ss{p \leq x \\ |g_{\lambda_x}(p)/s_x| \leq 1}} \frac{(g_{\lambda_x}(p)/s_x)^2}{p} \leq M \sum_{\ss{1 < t \leq x \\ p|t \Rightarrow p \in \mc{T}_x}} \frac{1}{t} \ll M \left(\exp\left(\sum_{\ss{p \leq x \\ p \in \mc{T}_x}} \frac{1}{p}\right) - 1\right) \ll M^2, 
\end{align}
where in the last step we used the fact that the argument in the exponential is bounded by $M \leq 1/2$. It remains to consider the discrepancy caused by replacing $a_x$ by $\tilde{a}_{x/t}$ in $V_{2,2}$. By the Cauchy-Schwarz inequality, for each $t$ in play,
\begin{align*}
|a_x-\tilde{a}_{x/t}|^2 \leq \left(\sum_{\ss{x/t < p \leq x\\ |g_{\lambda_x}(p)| \leq s_x}} \frac{(|g_{\lambda_x}(p)|/s_x)^2}{p}\right) \sum_{\ss{x/t < p \leq x \\ p \notin \mc{T}_x}} \frac{1}{p} \leq M \sum_{\ss{x/t < p \leq x \\ p \notin \mc{T}_x}} \frac{1}{p}.
\end{align*}
Therefore, we have
$$
V_{2,2} \ll M^2 + M \sum_{\ss{1 < t \leq x \\ p|t \Rightarrow p \in \mc{T}_x}} \frac{1}{t} \sum_{\ss{x/t < p \leq x \\ p \notin \mc{T}_x}} \frac{1}{p}.
$$
We split the range $t \leq x^{1-M}$ from $x^{1-M} < t \leq x$. In the first range we use Mertens' theorem and \eqref{eq:TaylorExp} to obtain
$$
\leq M\sum_{\ss{1 < t \leq x^{1-M} \\ p|t \Rightarrow p \in \mc{T}_x}} \frac{\log(\log x/\log(x/t)) + o(1)}{t} \ll M \log(1/M) \sum_{\ss{1 < t \leq x \\ p|t \Rightarrow p \in \mc{T}_x}} \frac{1}{t} + o(1) \ll M^2 \log(1/M) + o(1).
$$
In the range $x^{1-M} < t \leq x$ we split the range over $p$ as $(x/t,x^M]$ and $(x^M, x]$. Dropping the condition $p \notin \mc{T}_x$ and using \eqref{eq:TaylorExp}, the latter range contributes
$$
\leq M\left(\sum_{\ss{x^{1-M} < t \leq x \\ p|t \Rightarrow p \in \mc{T}_x}} \frac{1}{t}\right)\left(\sum_{x^M<p\leq x} \frac{1}{p}\right) \ll M^2 \log(1/M).
$$
For the range $(x/t,x^M]$ we swap orders of summation. Note that since $t > x^{1-M}$, every integer $t = qm$, where $q \in \mc{T}_x$, $p'|m \Rightarrow p' \in \mc{T}_x$ and $P^+(m) \leq q$. The point here is that since $p \leq x^M$, the range $(x/p,x]$ for $q$ is now much larger, and Mertens' theorem is efficient. The contribution here is then bounded above by
$$
\leq M\sum_{p \leq x^M} \frac{1}{p}\sum_{\ss{x/p < q \leq x \\ q \in \mc{T}_x}} \frac{1}{q} \sum_{\ss{x/(pq) < m \leq x \\ p'|m \Rightarrow p' \in \mc{T}_x}} \frac{1}{m} \ll M^2 \sum_{p \leq x^M} \frac{1}{p} \sum_{\ss{x/p < q \leq x}} \frac{1}{q} \ll \frac{M}{\log x} \sum_{p \leq x^M} \frac{\log p}{p} \ll M^3.
$$
Combining these estimates together, we obtain therefore, that 
$$
V_2 \ll K^2M^2 + M^2\log(1/M) + o(1).
$$ 
Adding together the contributions from $V_1,V_2$ and $V_3$ leads to the claim.
\end{proof}
\end{appendix}

\section*{Acknowledgments}
We warmly thank Prof. Andr\'{a}s S\'{a}rk\"{o}zy for his inspiring problem paper \cite{Sarkozy}, which has been a source of continued stimulation for the author since the beginning of his Ph.D in 2014.

\section*{AI Declaration}
Claude (Anthropic) was used to check references and for proofreading purposes. The latter revealed several minor gaps in the arguments from a previous version. The author fixed all of these gaps without AI assistance. More generally, all mathematical arguments in this paper are the work of the author.

\bibliographystyle{plain}
\bibliography{ErdSark.bib}

\end{document}